\documentclass[journal, 12pt, a4paper, onecolumn]{IEEEtran}

\usepackage[cmex10]{amsmath}
\usepackage{amsthm, amsfonts, amssymb, bm}
\usepackage{array}
\usepackage{color}
\usepackage{epsfig}
\usepackage{algorithm, algorithmic}
\usepackage[active]{srcltx}
\usepackage{subfig}
\usepackage{hyperref}
\usepackage{environ}
\usepackage{multirow}

\newtheorem{theorem}{Theorem}[section]
\newtheorem{lemma}[theorem]{Lemma}

\newtheorem{example}[theorem]{Example}
\newtheorem{remark}[theorem]{Remark}

\newtheorem{definition}[theorem]{Definition}

\def\norm#1{\|#1\|}

\def\R{{\mathbb R}}

\def\be{\begin{eqnarray*}}
\def\ee{\end{eqnarray*}}
\def\beq{\begin{equation}}
\def\eeq{\end{equation}}

\def\2q{\quad\quad}

\def\E{{\bf E}}

\def\L{{\mathbb L}}

\def\R{{\mathbb R}}

\def\:{{\,:\,}}

\def\norm#1{{\left\|\,#1\,\right\|}}
\def\abs#1{{\left|\,#1\,\right|}}

\def\norm #1{\|#1\|}
\def\abs #1{|#1|}
\def\argmax{\mathop{\rm arg\,max}}
\def\argmin{\mathop{\rm arg\,min}}
\def\define{:=}
\def\inprod#1#2{\langle #1,\,#2\rangle}

\def\dom{\hbox{dom}}

\begin{document}
\title{Bregman-divergence-guided Legendre exponential dispersion model with finite cumulants (K-LED)}

\author{Hyenkyun Woo
\thanks{School of Liberal Arts, Korea University of Technology and Education, 
hyenkyun@koreatech.ac.kr, hyenkyun@gmail.com}}


\maketitle

\begin{abstract}
Exponential dispersion model is a useful framework in machine learning and statistics. Primarily, thanks to the additive structure of the model, it can be achieved without difficulty to estimate parameters including mean. However, tight conditions on cumulant function, such as analyticity, strict convexity, and steepness, reduce the class of exponential dispersion model. In this work, we present relaxed exponential dispersion model K-LED (Legendre exponential dispersion model with K cumulants). The cumulant function of the proposed model is a convex function of Legendre type having continuous partial derivatives of K-th order on the interior of a convex domain. Most of the K-LED models are developed via Bregman-divergence-guided log-concave density function with coercivity shape constraints. The main advantage of the proposed model is that the first cumulant (or the mean parameter space) of the $1$-LED model is easily computed through the extended global optimum property of  Bregman divergence. An extended normal distribution is introduced as an example of 1-LED based on Tweedie distribution. On top of that, we present $2$-LED satisfying mean-variance relation of quasi-likelihood function. There is an equivalence between a subclass of quasi-likelihood function and a regular $2$-LED model, of which the canonical parameter space is open. A typical example is a regular $2$-LED model with power variance function, i.e., a variance is in proportion to the power of the mean of observations. This model is equivalent to a subclass of beta-divergence (or a subclass of quasi-likelihood function with power variance function). Furthermore, a new parameterized K-LED model is proposed. The cumulant function of this model is the convex extended logistic loss function which is generated by extended log and exp functions. The proposed model includes Bernoulli distribution and Poisson distribution depending on the selection of parameters of the convex extended logistic loss function.
\end{abstract}


\section{Introduction}
Various probability distributions, such as normal distribution, Poisson distribution, gamma distribution, and Bernoulli distribution, are formulated into the exponential families~\cite{barndorff14,brown86,mccullagh89} with sufficient statistics by virtue of the Fisher-Neyman factorization theorem~\cite{halmos49}. As a consequence of the additive structure of the exponential families, it is easy to estimate parameters, such as mean and variance, of probability distributions. Numerous applications of the exponential families are introduced in \cite{banerjee05,kulis12,murphy12,paul13,wainwright08}. For instance, \cite{banerjee05} introduce a mixture model with regular exponential families which has an equivalence with a subclass of  Bregman divergence~\cite{bregman67}. Though the exponential families have useful additive structure, the class of these is restricted, due to strong assumptions on cumulant function in terms of shape constraints of the distribution, such as analyticity, strict convexity, and steepness. Recently, the log-concave density estimation method is introduced~\cite{cule10,samworth18}. This method is a typical non-parametric estimation method with a simple coercivity shape constraint, and thus leads to relatively accurate density estimation results in a lower-dimensional space. See \cite{samworth18,saumard14} for more details and related applications. 

In this work, we are interested in a relaxation of the parameterized shape constraints of exponential dispersion model~\cite{jorgensen97}, i.e., natural exponential families with an additional dispersion parameter. Inspired from \cite{kass97}, we propose a relaxed exponential dispersion model which has  K continuously differentiable convex cumulant function of Legendre type (K-LED: Legendre exponential dispersion model with K cumulants). The proposed K-LED model is established through the parameterized log-concave density function based on Bregman divergence associated with a convex function of Legendre type (or Legendre). The main advantage of the proposed model is that, by the extended global optimum property of Bregman divergence~\cite{amari00, banerjee05},  the parameterized log-concave density function, which is developed via Bregman divergence associated with Legendre, becomes the $1$-LED model having the well-defined first cumulant (or the mean parameter space). In Section \ref{sec3}, we study in details on the construction of the $K$-LED model based on the parameterized log-concave density function. For more details on the various properties of  Bregman divergence, $f$-divergence, and various related equivalence in machine learning including classification, see \cite{dikmen15,reid10,reid11}. For the clustering (or segmentation) with Bregman divergence or generalized divergence, see \cite{banerjee05,lecellier10,nielsen09,paul13,teboulle07}. 

There are probability distributions having special conditions between mean and variance, such as quadratic variance function~\cite{morris82} and power variance function~\cite{bar-lev86,jorgensen97,tweedie84}. Let $\mu$ and $var(b)$ be mean and variance of observations $b$, then only six probability distributions (normal, Poisson, gamma, binomial, negative binomial, and generalized hyperbolic secant) of exponential dispersion models have the quadratic variance function~$var(b) =  a_1\mu^2 + a_2\mu +a_3$ where $a_i \in \R$ is a constant. See also \cite{pistone99}, for the generalized quadratic variance, known as finitely generated cumulants via a recurrence relation of polynomial between the first and the second cumulant. Although it is not a standard probability distribution having an analytic cumulant generating function, there is a relaxed (quasi-)probability distribution defined only by mean and variance; $p(b;\mu) = p_0(b)\exp(Q(b;\mu))$ with a quasi-likelihood function 
\begin{equation}\label{qlike}
Q(b;\mu) = -\int_{\mu}^b \frac{b-x}{\sigma^2V(x)}dx
\end{equation} 
where $\sigma^2>0$ is a dispersion parameter and $V(x)$ is a unit variance function satisfying mean-variance relation $var(b) = \sigma^2 V(\mu)$~\cite{mccullagh89,wedderburn74}.  Instead of mean-variance relation, by using a relation between the first and the second cumulant, the $2$-LED model is constructed. The equivalence between a subclass of quasi-likelihood function and the regular $2$-LED model satisfying the mean-variance function is studied in Section \ref{secQ}. See also \cite{kass97} for more details. A typical example of $2$-LED is Tweedie distribution~\cite{tweedie84} having {\it power variance function}:  
\begin{equation}\label{powerlaw}
var(b) =  \sigma^2 \mu^{2-\beta}
\end{equation} 
This distribution includes various probability distributions, such as normal distribution ($\beta=2$), Poisson distribution ($\beta=1$), compound Poisson-gamma distribution ($0<\beta<1$), gamma distribution ($\beta=0$), inverse Gaussian distribution ($\beta=-1$). Note that inverse Gaussian distribution is a  non-regular exponential dispersion model having a non-open canonical parameter space. Interestingly, on the boundary of the canonical parameter space, this distribution becomes Levy distribution which does not have the corresponding mean parameter space~\cite{barndorff14,jorgensen97}. Thus, the structure of Tweedie distribution is rather complicated. Besides, because of the analyticity of the cumulant generating function (or moment generating function) and the requirement of~\eqref{powerlaw} at the same time, the classic Tweedie distribution is not in exponential dispersion model when $\beta \in (1,2) \cup (2, +\infty)$~\cite{bar-lev86}. These strict constraints are relaxed in the proposed K-LED model with \eqref{powerlaw}.

Concerning $\beta$ of $\beta$-divergence (or $\beta$ in \eqref{powerlaw} of Tweedie distribution), it is not easy to directly use Tweedie distribution for the estimation of $\beta$ since it is not defined for all $\beta \in \R$. Recently, \cite{dikmen15} proposed an augmented exponential dispersion model (EDA). By an additional augmentation function, the domain of EDA is moved away from the boundary of the domain of the classic Tweedie distribution, and thus it is possible to estimate $\beta \in \R$ in a more natural way. However, the domain of EDA is limited to positive region, and thus the applicability of the model is reduced. The $\beta$-divergence with $\beta \in (1,2)$ has several interesting applications.  A typical one is that $\beta$-divergence with this region is used as robustified Kullback-Leibler divergence~\cite{basu98,eguchi01}. It gives a robust distance between two probability distributions. For instance, it was used for a robust spatial filter of the noisy EEG  data~\cite{samek13}. For more details on robustness of $\beta$-divergence, see \cite{basu98,cichocki10,eguchi01,woo17}. Moreover, in our previous works, this region is used for cutting-edge classification models; (1) H-Logitron~\cite{woo19} having high-order hinge loss with stabilizer. (2) The Bregman-Tweedie classification model~\cite{woo19b} which is developed by Bregman-Tweedie divergence (see also~\eqref{tweedieBregman} in Appendix). This classification model is an unbounded extended logistic loss function, including unhinge loss function~\cite{rooyen15}. Besides, the convex extended logistic loss function, which is between the logistic loss and the exponential loss, is an analytic convex function of Legendre type and thus can be used as a cumulant function of the K-LED model, which connect between Bernoulli distribution and Poisson distribution. The details are studied in Section \ref{berpoi}. Last but not least, the extended logistic loss function is composed of the extended elementary functions, that is, extended exponential function and extended logarithmic function. For more details on these functions and related applications in machine learning, see \cite{woo17,woo19,woo19b} and Appendix. 

The article is organized as follows. Section \ref{sec2} summarizes various properties of Legendre (or a convex function of Legendre type) and Bregman divergence associated with Legendre, which is essential ingredients in the following Sections. Section \ref{sec3} introduces the K-LED model, i.e., Legendre exponential dispersion model with K cumulants. This model is developed by  Bregman divergence associated with Legendre. The proposed Bregman-divergence-guided K-LED model inherently has the first cumulant, and thus it has the corresponding mean parameter space. For more details on the fundamental structure of exponential families, including exponential dispersion model, see~\cite{barndorff14,brown86,wainwright08}. Section \ref{secQ} studies the connection between the $2$-LED model and quasi-likelihood function based on mean-variance relation $var(b) = \sigma^2 V(\mu)$. Also, we introduce the $2$-LED model with power variance function~\eqref{powerlaw}, and the K-LED model, a cumulant function of which is the convex extended logistic loss function. We give our conclusions in Section \ref{sec6}.

\subsection{Notation}
Let $\inprod{a}{d} = \sum_{i=1}^n a_id_i$ where $a = (a_1,...,a_n) \in \R^n$ and $d = (d_1,...,d_n) \in \R^n$. $\R_+ = \{ x \in \R \;|\; x \ge 0 \}$, $\R_{++} = \{ x \in \R \;|\; x > 0 \}$, $\R_{-} = \{ x \in \R \;|\; x \le 0 \}$, and $\R_{--} = \{ x \in \R \;|\; x < 0 \}$. ${\mathbb Z}$ is a set of integer and ${\mathbb Z}_+ = \{0,1,2,... \}$. From \cite{woo17}, $\R$ is classified as
$$
\boxed{
\left\{
\begin{array}{l}
\R_e = \{ 2k/(2l+1) \;|\; k,l \in \mathbb{Z} \}\\ 
\R_o = \{ (2k+1)/(2l+1) \;|\; k,l \in \mathbb{Z} \}\\
\R_x = \R \setminus (\R_e \cup \R_o)
\end{array}
\right.
}
$$ Integration, multiplication, and division are performed component-wise. $\hbox{conv}{\cal B}$ is all convex combinations of the elements of a set ${\cal B}$. 

For a function $f$, $f \in C^K(\Theta)$ means that $f$ has continuous partial derivatives of $K$-th order on a convex set $\Theta \subseteq \R^n$. Let $f : \dom f \subseteq \R^n \rightarrow \R$ be a lower semicontinuous, convex, and proper function. Then the domain of $f$ is defined as
$\dom f = \{ x \in \R^n \;|\; f(x) < +\infty \}$.
This is known as the {\it effective domain}~\cite{roc70}. In this work, we always assume that $\dom f$ is a convex set, irrespective of convexity of $f$. Note that $int(\dom f)$ is the interior of $\dom f$ and $ri(\dom f)$ is the interior of $\dom f$ relative to its affine hull, the smallest affine set including $\dom f$. Hence, the relative interior $ri(\dom f)$ coincides with $int(\dom f)$ when the affine hull of $\dom f$ is $\R^n$. For this reason, we assume $ri(\dom f) = int(\dom f)$, unless otherwise stated. $cl(\dom f)$ is the closure of $\dom f$. $bd(\dom f) = cl(\dom f) \setminus int(\dom f)$ is the boundary of $\dom f$. As observed in \cite{hir96}, the convexity of $f$ can be extended to $\R^n$ by using the extended-valued real number system $\R_{+\infty} = \R\cup \{ +\infty \}$. That is, $f^e  : \R^n \rightarrow \R_{+\infty}$: $f^e(x) = \left\{\begin{array}{l} f(x) \quad x \in \Omega \\ +\infty \quad\; x \not\in \Omega  \end{array}\right.$
where $\Omega = \dom f$ or any other convex set for various purpose. If $f$ is not convex then we use the extended-valued real number system $\R_{\pm \infty} = \R \cup \{\pm\infty\}$. See \cite{woo19} for arithmetical operations in $\R_{\pm\infty}$. Let $\E[b]$ be the expectation of observations $b$. For simplicity, we use $\mu = \E[b]$.
\section{Preliminaries\label{sec2}}
This Section introduces some useful properties of a convex function of Legendre type, the corresponding Bregman divergence, and log-concave density functions. For more details on these, see \cite{amari16,barndorff14,bauschke97,hir96,roc70} and reference therein.

\subsection{A convex function of Legendre type}

\begin{theorem}\label{int}
Let $f: \dom f \rightarrow \R$ be lower semicontinuous, convex, and proper function on $\dom f \subseteq \R^n$. Then $f$ satisfies the following relation
\begin{equation}\label{increl}
int(\dom  f) \subseteq \dom \partial f \subseteq \dom  f
\end{equation}
where $\dom  f$ is a convex set and thus $int(\dom  f)$ is also convex~\cite[Th 6.2]{roc70}. Note that
$\dom \partial f = \{ x \in \dom f \;|\; \partial f(x) \not= \emptyset \}$
where $\partial f(x) = \{ x^* \;|\; f(z) \ge f(x) + \inprod{x^*}{z-x}, \forall z \}$ is a subgradient of $f$ at $x \in \dom f$.
\end{theorem}
As noticed in \cite{roc70}, $\dom \partial f$ is not necessarily convex, though $\dom f$ and $int(\dom f)$ are convex. Now, we define a convex function of Legendre type~\cite{bauschke97,roc70}. 
\begin{definition}\label{legendredef}
Let $f: \dom f \rightarrow \R$ be lower semicontinuous, convex, and proper function on $\dom f \subseteq \R^n$. Then $f$ is a convex function of Legendre type (or Legendre), if the following conditions are satisfied.
\begin{itemize}
\item $int(\dom f)\not= \emptyset$ and $f \in C^1(int(\dom f))$
\item $f$ is strictly convex on $int(\dom f)$  
\item ({\it steepness}) $\forall x \in bd(\dom f)$ and $\forall y \in int(\dom f),$
\begin{equation}\label{steep}
\lim_{t \downarrow 0} \inprod{\nabla f(x + t(y-x))}{y-x} = -\infty
\end{equation}
\end{itemize}
For simplicity, let us denote a class of convex functions of Legendre type as
$$
\boxed{\mathbb{L}_n = \{ f: \dom f \subseteq \R^n \rightarrow \R \;|\; f \hbox{ is Legendre} \}}
$$
\end{definition}
Here, \eqref{steep} is known as the steepness condition in statistics~\cite{barndorff14,brown86}. The following Theorem~\cite{bauschke97,roc70} is useful while we characterize Legendre exponential dispersion model with $K$ cumulants (K-LED).
\begin{theorem}\label{isoL}
$f \in \L_n$ if and only if $f^* \in \L_n$, where  $f^*(x) = \sup_{t} \inprod{x}{t} - f(t)$ is the conjugate function of $f$.  The corresponding gradient 
\begin{equation}\label{toISO}
\nabla f : int(\dom  f) \rightarrow int(\dom  f^*) : x \rightarrow \nabla f(x)
\end{equation}
is a topological isomorphism with inverse mapping 
$(\nabla f)^{-1} = \nabla f^*$. 
\end{theorem}

The coercivity of Legendre is useful while we characterize Tweedie distribution~\cite{bar-lev86,jorgensen97,tweedie84} and log-concave density functions~\cite{cule10}. 
\begin{theorem}\label{Thcoercivity}
Let $f \in \L_n$, then the followings are equivalent:
\begin{enumerate}
\item $f$ is coercive, i.e., $\lim_{\norm{x}\rightarrow +\infty} f(x) = +\infty$
\item There exists $(a_1,a_2) \in \R_{++} \times \R$ such that $f(x) \ge a_1\norm{x} + a_2$, for all $x \in \R^n$
\item $0 \in int(\dom f^*)$ 
\item $\int_{\R^n} \exp(-f(x)) dx < +\infty$ 
\end{enumerate}
\end{theorem}
For the proof of Theorem \ref{Thcoercivity}, see \cite[Th. 6.1]{barndorff14} and \cite[Prop. 14.16]{bauschke11}. 
\begin{definition}\label{logconv}
Let $f \in \L_n$ and $\int_{\R^n} g(x) \nu(dx) = 1$, where $g(x) = \exp(-f(x))$ and $\nu(dx)$ is an appropriate continuous Lebesgue (or discrete  counting) measure on $\R^n$. Then $g$ is a log-concave (probability) density function.
\end{definition}
See \cite{barndorff14,bauschke11, brown86, hir96, roc70} for other useful properties of convex functions and their applications in statistics. 

\subsection{Bregman divergence associated with Legendre 
}
Consider Bregman divergence associated with~$f \in \L_n$:
\begin{equation}\label{bregx}
D_{ f}(x|y) =  f(x) -  f(y) - \inprod{x-y}{\nabla  f(y)} 
\end{equation}
where $(x,y) \in \dom f  \times int(\dom f)$ and $D_{ f}(x|y) \in  \R_+$. It also is formulated with the conjugate function $f^* \in \L_n$ as  $D_{ f}(x|y) =  f(x) +  f^*(\nabla f(y)) - \inprod{x}{\nabla f(y)}$. As observed in information geometry~\cite{amari00,amari16}, Bregman divergence~\eqref{bregx} is related to the canonical divergence. Actually, it includes various divergences; (1) Itakura-Saito divergence $D_{ f}(x|y) = \left(\frac{x}{y}\right)-\log\left(\frac{x}{y}\right)-1$ with $ f(x) = -\log x$. This divergence is induced from gamma distribution~\cite{fevotte09,woob16}. (2) Generalized Kullback-Leibler divergence (I-divergence) $D_{ f}(x|y) = x\log\left(\frac{x}{y}\right)-(x-y)$ with $ f(x) = x\log x$. This generalized distance is induced from Poisson distribution~\cite{dias10b,teboulle07}. (3) $norm^2$-distance $D_{ f}(x|y) = \frac{1}{2}(x-y)^2$ with $ f(x) = \frac{1}{2}x^2$. This distance can be easily derived from normal distribution. 

In addition, we summarize several useful properties of Bregman divergence associated with Legendre. See \cite{bauschke97,teboulle07} for more details. 
\begin{theorem}\label{legendreth}
Let $f \in \L_n$ and $f^* \in \L_n$. Then Bregman divergence associated with $f$ satisfies the following properties.
\begin{enumerate}
\item $D_{f}(x|y)$ is strictly convex with respect to $x$ on $int(\dom f)$.
\item $D_{f}(x|y)$ is coercive with respect to $x$, for all $y \in int(\dom  f)$. 
\item $D_{f}(x|y)$ is coercive with respect to $y$, for all $x \in int(\dom  f)$ if and only if $\dom  f^*$ is open.
\item $D_{f}(x | y)=0$ if and only if $x=y$\;\; where $y \in int(\dom f)$
\item For all $x,y \in int(\dom f)$, $D_{f}(x|y) = D_{f^*}(\nabla f(y)|\nabla f(x))$
\item (Global optimum property~\cite{banerjee05}) Let $x \in \dom f$ and $\E(x) \in int (\dom f)$, then for all $(x,y) \in \dom f \times int(\dom f)$, we have $\E(D_{f}(x | y)) \ge  \E(D_{f}(x|\E(x)))$. 
\end{enumerate}
\end{theorem}
The global optimum property in Theorem \ref{legendreth} (6) is satisfied, irrespective of the convexity of Bregman divergence in terms of second variable. However, if Bregman divergence has an additional regularization term on the second variable, this property does not satisfied anymore. See \cite{lecellier10,paul13,rudin92,woo17} for more details on the regularized Bregman divergence and its applications in image processing. The following canonical divergence (reformulated Bregman divergence associated with Legendre) is helpful for the characterization of Bregman-divergence-based probability distribution and its relation with the K-LED model.
\begin{theorem}[Extended global optimum property~\cite{amari00}]
\label{canTH}
Let $f \in \L_n$, $f^* \in \L_n$, and $x \in \dom f$. Then the following canonical divergence
\begin{equation}\label{canD}
\boxed{
d_f(x;\theta) \define  D_f(x|\nabla f^*(\theta)) = f(x) + f^*(\theta) -\inprod{x}{\theta}
}
\end{equation}
is strictly convex with respect to the canonical parameter $\theta \in int(\dom f^*)$. Consider the minimization problem:
$$ 
\hat{\theta} = \argmin_{\theta \in int(dom f^*)} d_f(x;\theta)
$$ 
The solution of the above minimization problem exists within an extended-valued real number system~$\R_{\pm\infty}$
$$
\hat{\theta} = \left\{\begin{array}{l} 
\nabla f(x) \quad \hbox{ if } x \in int(\dom f)\\
\pm\infty \qquad \hbox{ if } x \in \dom f \setminus int(\dom f) 
\end{array}\right.
$$
Additionally, the global optimum property in Theorem \ref{legendreth} (6) is extended to $\E[x]\in\dom f$.
\end{theorem} 
\begin{proof}
Let $x \in int(\dom f)$ then it is trivial that $\hat{\theta} = \nabla f(x)$. Consider $x \in \dom f \setminus int(\dom f)\; ( \not= \emptyset)$. Let $\hat{x}_m = (1-m^{-1})x + m^{-1}x_m \in int(\dom f)$ where $x_m \in int(\dom f)$ and $x_m \rightarrow x$ as $m \rightarrow +\infty$. Then $\hat{\theta}_m = \nabla f(\hat{x}_m)$. Thus, as $m \rightarrow \infty$, $\hat{x}_{\infty} = x   \not\in int(\dom f)$ and $\hat{\theta}_{\infty}= \hat{\theta} = +\infty$ (or $-\infty$). Regarding the global optimum property, from $\E(d_f(x;\theta)) = d_f(\E(x);\theta) + \E(f(x)) - f(\E(x))$, we have  $\theta_f =\argmin_{\theta}\E(d_f(x;\theta))$ where $\theta_f = \nabla f (\E(x))$ for all $\E(x) \in int(\dom f)$ and $\theta_f = \pm\infty$ for all $\E(x) \in \dom f \setminus int(\dom f)$.
\end{proof}
Theorem \ref{canTH} is useful while we analyze Tweedie distribution, such as zero-inflated compound Poisson-gamma distribution ($\beta \in (0,1)$) and inverse Gaussian distribution ($\beta=-1$). The following example shows that it is possible to build up a parameterized log-concave density function with a mean parameter space which is induced from Bregman divergence associated with Legendre. 
\begin{example}[Bregman-divergence-guided log-concave density function]\label{ex8}
Assume that observations $b \in \dom f$, $f \in \L_n$, $\Omega = int(\dom f)$, $\Omega^c = \dom f \setminus \Omega$, and $\Omega^*=\dom f^* = int(\dom f^*)$. From Theorem \ref{legendreth} (2), it is easy to check the coercivity of Bregman divergence associated with Legendre $d_{f}(b;\theta)$ for all $\theta \in \dom f^*$. Hence, we have the corresponding log-concave density function (see Definition \ref{logconv}).
\begin{equation}\label{logC}
p_{f}(b;\theta) = \exp(- d_{f}(b;\theta)) p_0(b) = \exp(\inprod{b}{\theta} - f^*(\theta)) p_1(b)
\end{equation}
where $p_1(b) = p_0(b)\exp(-f(b))$ is a base measure with an appropriate $p_0(b)$ satisfying $\int_{\cal B} p_{f}(b;\theta) \nu(db) = 1$. Consider the corresponding likelihood function $\ell(\theta;b_1,...,b_M) = \prod_{i=1}^M p_f(b_i;\theta)$ and a minimization problem with the negative log-likelihood function:
\begin{eqnarray}\label{minX}
\hat{\theta}_{avg} &=& \argmin_{\theta \in dom f^*} \;\sum_{i=1}^M d_f(b_i;\theta) = M d_f(b_{avg};\theta) + h(b)
\end{eqnarray}
where $b_{avg} = \frac{1}{M}\sum_{i=1}^M b_i \in \dom f$ and $h(b) = \sum_{i=1}^M f(b_i) - M f(b_{avg})$. 
Due to Theorem \ref{canTH}, the solution of \eqref{minX} becomes
$$
\hat{\theta}_{avg} = \left\{\begin{array}{l} 
\nabla f(b_{avg}) \quad \hbox{ if } b_{avg} \in \Omega\\
\pm\infty \qquad\quad \hbox{ if } b_{avg} \in \Omega^c 
\end{array}\right.
$$
Since $f \in \L_n$, we have a bijective map
$$
\nabla f :  \Omega \cup \Omega^c \rightarrow \Omega^* \cup \{\pm\infty \}   
$$  
where $\dom f = \Omega \cup \Omega^c$ is the mean parameter space and $\Omega^*$ is the canonical parameter space. See also \cite{amari00}. In fact, \eqref{logC} becomes a natural exponential family with the mean parameter space $\dom f$ if $f^*$ is a cumulant function.
\end{example}

\subsection{$\beta$-divergence and quasi-likelihood function with power variance function\label{betasec}}
Compared to Bregman divergence associated with Legendre, $\beta$-divergence~\cite{basu98,eguchi01} is less structured and tightly connected to quasi-likelihood function~\eqref{qlike} with power variance function~\eqref{powerlaw}. Formally, $\beta$-divergence $D_{\beta} : \Omega_L \times \Omega_R \rightarrow \R_+$ is defined as
\begin{equation}\label{betaDiv}
D_{\beta}(b|u) = \int_u^b x^{\beta-2}(b-x)dx = 
\left\{\begin{array}{l} 
\left(\frac{b}{u}\right) - \ln\left(\frac{b}{u}\right) -  1 , \hskip 2.8cm \hbox{ if } \beta = 0\\
b \ln\left(\frac{b}{u}\right) - (b-u), \hskip 2.7cm \hbox{ if } \beta = 1\\
\frac{b}{\beta-1}(b^{\beta-1}-u^{\beta-1}) - \frac{1}{\beta}(b^{\beta}-u^{\beta}), \quad \hbox{otherwise}
\end{array}\right.
\end{equation}
where $\Omega_L \times \Omega_R = \{ (b,u) \in \R\times\R \;|\; D_{\beta}(b|u) \in \R_+ \}$ is the domain of $\beta$-divergence. As observed in \cite{woo17}, non-negativeness of the range of $\beta$-divergence is guaranteed under assumption $x^{\beta-2} \ge 0$. Note that the mathematical formulation of quasi-likelihood function with power variance function~\eqref{powerlaw} is equal to that of $\beta$-divergence, i.e., $D_{\beta}(b|\mu) = - \sigma^2 Q(b;\mu)$
where $Q(b;\mu) = -\int_{\mu}^b \frac{b-x}{\sigma^2V(x)}dx$. However, due to an additional condition on the variance function $V(x) = x^{2-\beta} \ge 0$, the equivalence is not always true on the domain $\Omega_L \times \Omega_R$ of $\beta$-divergence. As observed in \cite{woo17}, $\beta$-divergence can be reformulated to  Bregman-beta divergence~\eqref{betaBregman} under restriction of the domain to $\dom\Phi \times int(\dom\Phi)$ where $\Phi$ is defined in~\eqref{basefn}. In fact, the regular Tweedie distribution~\cite{jorgensen97,tweedie84} is developed based on Bregman-beta divergence~\eqref{betaBregman}. The details are dealt with in Section \ref{secQA}.

\section{K-LED: Legendre exponential dispersion model with K cumulants and $K \ge 1$\label{sec3}}
This Section presents the K-LED model, Legendre exponential dispersion model with K cumulants, derived from Bregman divergence associated with Legendre in~\eqref{canD}.
    
Let us start with natural exponential families~\cite{barndorff14,jorgensen97,wainwright08}:
\begin{equation}\label{nef}
p_{\psi}(b;\theta) = \exp(\inprod{b}{\theta} - \psi(\theta))p_1(b)
\end{equation}    
where $b \in {\cal B} \subseteq \R^n$ is an observation (or a random vector) and $\theta \in \Theta = \{ \theta \in \R^n \;|\; \psi(\theta) < +\infty \}$ is a canonical parameter. For all $\theta \in \Theta$, if \eqref{nef} is uniquely determined then it is full. Note that it is regular if $\Theta$ is open, and non-regular if $\Theta$ is not open. Here, ${\cal B}$ is a set of random vectors and $B = cl(\hbox{conv} {\cal B})$. The minimal condition of \eqref{nef} means that $\inprod{b}{\theta}$ is not constant for any non-zero $\theta$. When $b$ is replaced by a sufficient statistic $\zeta(b)$, \eqref{nef} becomes the traditional exponential families. 
For simplicity, we only consider exponential dispersion model~\cite{jorgensen97} (i.e., natural exponential families with an additional dispersion parameter). 

From $\int_{{\cal B}} p_{\psi}(b;\theta) \nu(db) = 1$, we have 
\begin{equation}\label{PSIx}
\psi(\theta) = \log \int_{\cal B} \exp(\inprod{b}{\theta})p_1(b)\nu(db),
\end{equation}
where $\nu(db)$ is an appropriate  continuous Lebesgue (or discrete counting) measure depending on ${\cal B}$. Note that $\psi$ in \eqref{PSIx} is a cumulant function (or log-partition function) of $p_{\psi}$ and analytic on its interior of the domain. Additionally,  under the minimality condition of \eqref{nef},  it is not difficult to show that $\psi(\theta)$ is strictly convex~\cite{brown86,wainwright08}. 
The main advantage of \eqref{nef} is that we can easily obtain mean (first cumulant), covariance (second cumulant), or even higher order cumulants of observations from the cumulant generating function 
$ 
K_{\psi}(x;\theta) = \log M_{\psi}(x;\theta) = \psi(\theta + x) - \psi(\theta)
$ 
where $\{x\} + \dom\psi \subseteq \dom\psi$ and $M_{\psi}(x;\theta) = \E(\exp(\inprod{x}{b})) =  \exp( \psi(\theta + x) - \psi(\theta))$ is the moment generating function.
For instance,
$ 
\nabla K_{\psi}(0;\theta) =  \nabla\psi(\theta) = \E(b)
$  
and 
$\nabla^2 K_{\psi}(0;\theta) = \nabla^2\psi(\theta) = cov(b)$, where $\theta \in int(\dom \psi)$. In this way, a cumulant generating function uniquely determines a probability distribution within minimal natural exponential families~\cite{barndorff14}.  Due to $\mu = \nabla\psi(\theta) \in int(\dom\psi^*)$  and the additional condition ${\cal B} \subseteq \dom\psi^*$~\cite{banerjee05}, $\dom\psi^*$ is known as the {\it mean parameter space} and $\dom\psi$ is known as the {\it canonical parameter space}~ \cite{amari00,wainwright08}. As described in \cite[Theorem 3]{banerjee05} and \cite[Theorem 9.1, 9.2]{barndorff14}, the mean parameter space $\dom\psi^*$ and a set of observations ${\cal B}$ satisfies the following condition
\begin{equation}\label{assumeKLED}
int(\dom\psi^*) = int(B) \quad \hbox{and} \quad {\cal B} \subseteq \dom\psi^* \subseteq B 
\end{equation}
where $B = cl(\hbox{conv}{\cal B})$.  In this work, we assume that \eqref{assumeKLED} is always true, unless otherwise stated. Note that, from \eqref{increl}, we have $int(\dom\psi) \subseteq \dom(\partial \psi)$. If $\dom(\partial\psi) \setminus int(\dom\psi) \not=  \emptyset$ then we can not obtain unique mean and variance on the set $\dom(\partial\psi) \setminus int(\dom\psi)$. Therefore, the condition  $int(\dom\psi) = dom(\partial\psi)$ is highly demanded. Actually, this is achieved through the steepness condition~\eqref{steep} of $\L_n$.  However, a function in $\L_n$ is not always analytic on its interior of the domain, and thus it may not become a cumulant function of minimal natural exponential families. For instance, consider a convex function $\psi(x) = x^{8/3}$ of Legendre type. The domain of this function is $\R$ and $\dom(\nabla^k \psi) = \R$ for all $k=0,1,2$. However, we have $\dom(\nabla^3 \psi)$ (= $\R_{++}$ or $\R_{--}$). There is domain inconsistency depending on the order of differentiability.  

As observed in Example \ref{ex8}, Bregman-divergence-guided log-concave density function  naturally has the first cumulant or the mean parameter space. Hence, if the conjugate of a base function of Bregman divergence satisfies mean-variance relation, such as power variance function~\eqref{powerlaw}, then we have a log-concave density function with the first and the second cumulants. In this way, we can build up the relaxed exponential dispersion model with finite cumulants. 
\begin{definition}[{\bf K-LED}]
Let $b \in {\cal B}$ be an observation with \eqref{assumeKLED}. Consider $\psi \in \L_n \cap C^K(int(\dom\psi))$ with $\nabla\psi(\theta) = \E[b]$ and $K \ge 1$. Then Legendre exponential dispersion model with K cumulants (K-LED) is defined as
\begin{equation}\label{EDx}
p_{\psi}(b;\theta,\sigma^2) = \exp\left(\frac{\inprod{b}{\theta} - \psi(\theta)}{\sigma^2}\right)p_1(b,\sigma^2)
\end{equation}
where $\sigma^2>0$ is a dispersion parameter, $\dom\psi$ is the canonical parameter space, and $\dom\psi^*$ is the mean parameter space. $p_1(b,\sigma^2)$ is a base measure satisfying $\int_{\cal B} p_{\psi}(b;\theta,\sigma^2)\nu(db)=1$. Note that $p_{\psi}$ is regular, if $\dom\psi$ is open and non-regular, if $\dom\psi$ is not open. In addition, $K_{\psi}(x;\theta,\sigma^2) = \sigma^{-2}(\psi(\theta + x\sigma^2) - \psi(\theta))$ is a cumulant generating function up to K-th cumulant. Here, $\theta + x \sigma^2 \in \dom\psi$, for all $\theta \in dom\psi$.
\end{definition}

\begin{remark}\label{rmk1}
\begin{enumerate}
\item[]
\item If $p_{\psi}(b;\theta,\sigma^2)$ is uniquely determined for all $\theta \in \dom\psi = \{ \theta \in \R^n \;|\; \psi(\theta) < +\infty \}$ then \eqref{EDx} is the full K-LED model for a given $\sigma^2$.
\item In case of $\nabla\psi(\theta_0) = \nabla^2\psi(\theta_0)=0$, the K-LED model~\eqref{EDx} is degenerate at $\theta_0 \in \dom\psi$.
\item  Let $x=b/\sigma^2$ then~\eqref{EDx} can be reformulated as $\exp(\inprod{x}{\theta} - \sigma^{-2}\psi(\theta))p_1(x,\sigma^2)$. This becomes the additive K-LED model corresponding to the classic additive exponential dispersion model~\cite{jorgensen97}. Additionally, let $\sigma^2$ in \eqref{EDx} be a constant, $\theta_1 = \theta/\sigma^2$, and $\psi_1(\theta_1) = \psi(\sigma^2\theta_1)/\sigma^2$. Then, we get a density function in natural exponential families:
\begin{equation}\label{varEDM}
p_{\psi_1}(b;\theta_1) = \exp(\inprod{b}{\theta_1} - \psi_1(\theta_1))\bar{p}_1(b).
\end{equation}
Then the mean and variance of the K-LED model~\eqref{EDx} are given as $\nabla\psi(\theta) = \E(b)=\nabla\psi_1(\theta)$ and $\sigma^2\nabla^2\psi(\theta) = \E[(b-\E(b))(b-\E(b))^T] = \nabla^2\psi_1(\theta)$, where $\psi \in C^2(int(\dom\psi))$. In this way, \eqref{varEDM} is known as a density function of the scaled exponential families. See~\cite{jiang12,kulis12} for more details on the scaled exponential families for Dirichlet process mixture model where $\sigma^2$ is regularized to control accuracy of the density estimation. 
\end{enumerate}
\end{remark}
As commented in Section \ref{betasec} and \cite{woo17}, there is partial equivalence between a subclass of $\beta$-divergence and quasi-likelihood function with power variance function~\eqref{powerlaw}. Hence, it is natural to consider the relation between Bregman divergence associated with Legendre and the proposed K-LED model~\eqref{EDx}. In case of K-LED, the minimum requirement is the existence of the first cumulant, i.e., the mean parameter space. As noticed in Example \ref{ex8}, it can be satisfied by Bregman divergence associated with Legendre.  
\begin{theorem}\label{openeq}
 Let $b \in {\cal B}  \subseteq \dom\psi^*$, $\psi^* \in \L_n$, and $\psi \in \L_n$. Assume that $\dom \psi^*$ and $\dom \psi$ are open. Then $p_{\psi}(b;\theta,\sigma^2) =  \exp(-d_{\psi^*}(b;\theta)/\sigma^2)p_0(b,\sigma^2)$ is a parameterized log-concave density function with the first cumulant $\nabla\psi(\theta) = \E[b] \in \dom\psi^*$ for all $\theta \in \dom\psi$. Here, $\sigma^2 > 0$ is a constant and $p_0(b,\sigma^2)$ is a base measure satisfying $\int_{\cal B} p_{\psi}(b;\theta,\sigma^2)\nu(db)=1$. That is, $p_{\psi}(b;\theta,\sigma^2)$ is the $1$-LED model.  
\end{theorem}
\begin{proof}
By Theorem \ref{legendreth} (2), it is easy to check coercivity of $d_{\psi^*}(b;\theta)$ with respect to $b$. Thus, $p_{\psi}(b;\theta,\sigma^2)= \exp(-d_{\psi^*}(b;\theta)/\sigma^2)p_0(b,\sigma^2)$ is a log-concave density function with an appropriate base measure $p_0(b,\sigma^2)$ satisfying $\int_{{\cal B}} p_{\psi}(b;\theta,\sigma^2)\nu(db) =1$ (see Theorem \ref{Thcoercivity} (4)). Regarding the existence of the first cumulant, from Theorem \ref{canTH}, we have 
\begin{eqnarray}\nonumber
\nabla\psi^*(\E(b)) &=& \argmin_{\theta \in dom \psi} \E(d_{\psi^*}(b;\theta))\\ \label{optx1}
&=& \argmin_{\theta \in dom \psi} d_{\psi^*}(\E(b);\theta) + h(b)
\end{eqnarray}
where $h(b) =  \E[\psi^*(b)] - \psi^*(\E(b))$. Hence, for any mean value $\mu=\E(b) \in \dom\psi^*$, there is always corresponding unique canonical parameter $\theta = \nabla\psi^*(\mu) \in \dom\psi$. 
\end{proof}
Typical examples of Theorem \ref{openeq} are gamma and normal distributions. For gamma distribution, $\psi^*(x) = -\log(x)$ and $\psi(x) = -\log(-x)-1$. For normal distribution, $\psi(x) = \psi^*(x) = \frac{1}{2}x^2$. An extension of normal distribution, having a constant unit variance function, is not uncomplicated. In the following example, we introduce an extended normal distribution via power variance function~\eqref{powerlaw}. This distribution is induced from Bregman-beta divergence $d_{\Phi}$~\eqref{betaBregman} or Bregman-Tweedie divergence $d_{\Psi}$~\eqref{tweedieBregman}, which were introduced in \cite{woo17,woo19b}. 
\begin{example}[Extended normal distribution]\label{TweedieX}
Let $b \in {\cal B} \subseteq dom\Phi$ and $\beta \in (1,\infty) \cap \R_e$. Then $\Phi,\Psi \in \L_1 \cap C^1(\R)$ where $\Phi(x) = \frac{1}{\beta(\beta-1)}x^{\beta}$ and $\Psi(y) = \frac{1}{\beta}((\beta-1)y)^{\frac{\beta}{\beta-1}}$. Therefore, we have $d_{\Psi}(b;\theta) = d_{\Phi}(\theta;b)$. Consider Bregman-beta divergence
$$
d_{\Phi}(b;\theta) = \Phi(b) + \Psi(\theta) - \inprod{b}{\theta}
$$
From $\dom\Phi = \dom\Psi = \R$, $\Phi$ and $\Psi$ are coercive and thus $\exp(-d_{\Phi}(b;\theta))$ is a log-concave density function. From Theorem \ref{openeq}, we have an $1$-LED model 
$$
p_{\Psi}(b;\theta,\sigma^2) = \exp(-d_{\Phi}(b;\theta)/\sigma^2)p_0(b,\sigma^2)
$$
with the first cumulant function $\nabla\Psi(\theta) = \E[b] \in \R$ where $\theta \in \R$, and a base measure $p_0(b,\sigma^2)$ satisfying $\int_{\cal B} p_{\Psi}(b;\theta,\sigma^2) \nu(db)=1$. When $\beta=2$, we obtain the famous normal distribution with $d_{\Phi}(b;\theta) = \frac{1}{2}(b-\theta)^2$. Hence, the $1$-LED model $p_{\Psi}$ is an extended normal distribution satisfying power variance function~\eqref{powerlaw}. 
When $\beta \not=2$, this distribution does not have the classic cumulant generating function~\cite{bar-lev86}. Actually, if $\beta \in (2,+\infty) \cap \R_e$ then $\Psi \in C^1(\R)$ and thus $p_{\Psi}$ is the $1$-LED model. However, if $\beta \in (1,2) \cap \R_e$ then $\Psi(\theta) \in C^K(\R)$ with $K \ge 2$ depending on $\beta$. Since $\nabla\Psi(0) = \nabla^2\Psi(0)=0$, as commented in Remark \ref{rmk1} (3), $p_{\Psi}$ is a degenerate $K$-LED model. See Figure \ref{fig:img1} for this extended normal distribution. 
\end{example}

\begin{figure*}[t]
\centering
\includegraphics[width=5.5in]{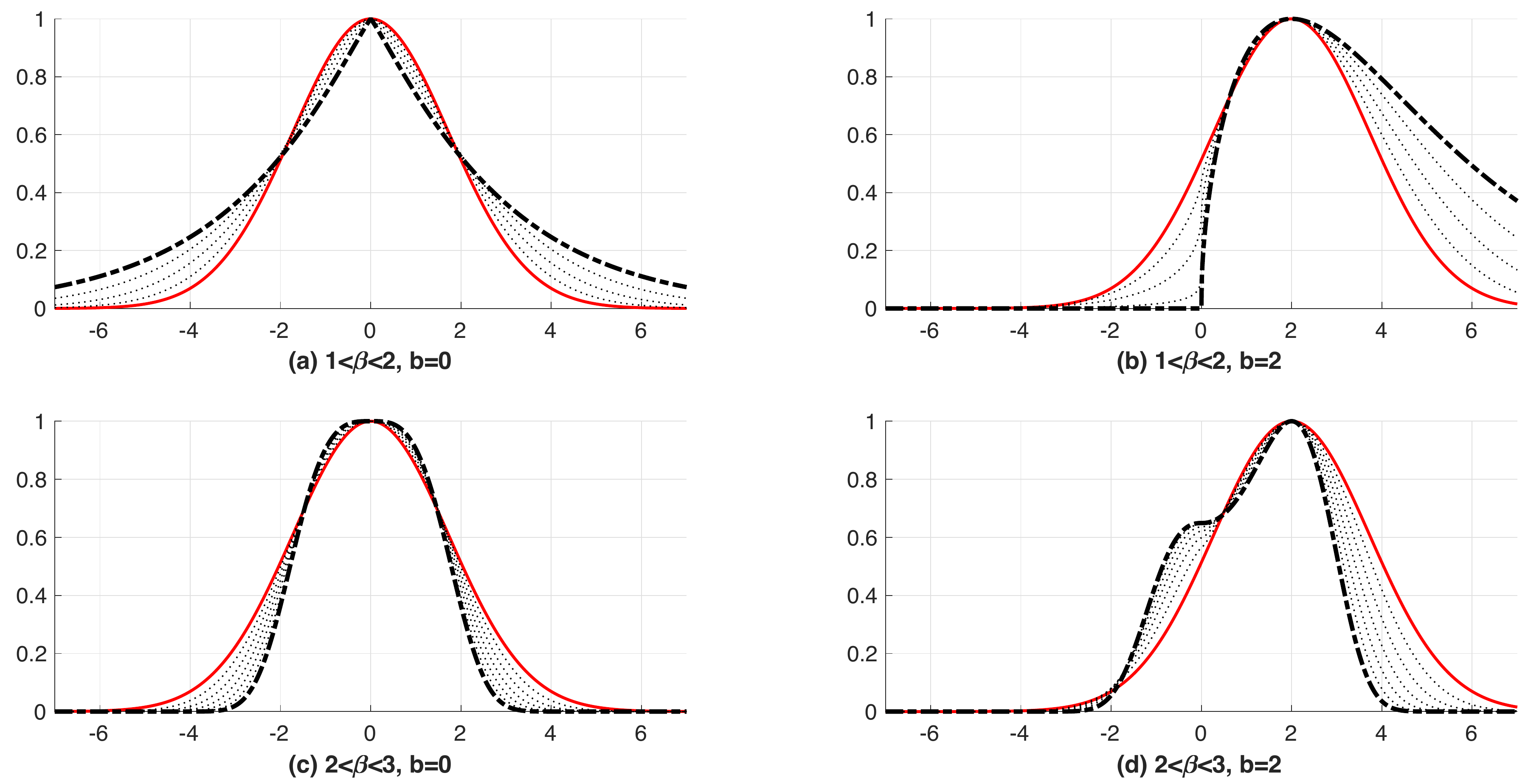} 
\caption{The graphs of $1$-LED with \eqref{powerlaw} and $\beta \in (1,\infty)\cap \R_e$. 
 The red curve is the normal distribution ($\beta=2$ and $\exp(-(b-\mu)^2/6)$). Others are extended normal distributions $\exp(-D_{\Psi}(b|\mu)/3)$ with respect to $\mu$. That is, (a) $b=0$ and $\beta=16/9,14/9,...10/9 \hbox{(dash-dot line)}$, (b) $b=2$ and $\beta=16/9,14/9,...,10/9 \hbox{(dash-dot line)}$, (c) $b=0$ and $\beta=30/9 \hbox{(dash-dot line)},28/9,...,20/9$, and (d) $b=2$ and $\beta=30/9 \hbox{(dash-dot line)},28/9,...,20/9$. For more details, see Example \ref{TweedieX}.
}
\label{fig:img1}
\end{figure*}

Now, consider the case that $\dom\psi^*$ or $\dom\psi$ are not open. In other words, the regular K-LED model with a non-trivial mean parameter space which is not open and the non-regular K-LED model, of which the canonical parameter space is not open. A typical example of the regular K-LED model is Bernoulli distribution having discrete random variables. That is, ${\cal B} = \{0,1\}$ and $B = cl\hbox{conv}{\cal B} = [0,1]$ (see \eqref{assumeKLED}). 
 \begin{theorem}\label{nopeneq1}
 Let $b \in {\cal B}$ with \eqref{assumeKLED}, $\psi^* \in \L_n$, $\psi \in \L_n$, and $\sigma^2 > 0$ be a constant. (a) Assume that $\dom\psi$ is open and $\dom\psi^*$ is not open. Then $p_{\psi}(b;\theta,\sigma^2) = \exp(-d_{\psi^*}(b;\theta)/\sigma^2)p_0(b,\sigma^2)$ is a parameterized log-concave density function with the first cumulant $\nabla\psi(\theta) = \E[b] \in \dom\psi^*$ in an extended-valued real number system. Here, $p_0(b,\sigma^2)$ is a base measure satisfying $\int_{\cal B} p_{\psi}(b;\theta,\sigma^2)\nu(db)=1$. Hence, $p_{\psi}(b;\theta,\sigma^2)$ is the $1$-LED model.  (b) Assume that $\dom \psi$ is not open and $\dom\psi^*$ is open. Then $p_{\psi}(b;\theta,\sigma^2) = \exp(-d_{\psi}(\theta;b)/\sigma^2)p_0^{nr}(b,\sigma^2)$ is a parameterized log-concave density function with the first cumulant $\nabla\psi(\theta) = \E[b] \in \dom\psi^*$ only when $\theta \in int(\dom\psi)$. That is, $p_{\psi}(b;\theta,\sigma^2)$ is a non-full $1$-LED model with $\theta \in int(dom\psi)$. Here, $p^{nr}_0(b,\sigma^2)$ is a base measure satisfying $\int_{\cal B} p_{\psi}(b;\theta,\sigma^2)\nu(db)=1$. 
\end{theorem}
\begin{proof}
(a) In Theorem \ref{openeq}, it is treated the case $int(B)$ and $\dom\psi$ is open. Consider ${\cal B} \not\subset int(B)$ and let $\Omega_1 = \dom\psi^* \setminus int(B) \not= \emptyset$. Since $\dom\psi$ is open, it is not difficult to see that $p_{\psi}(b;\theta;\sigma^2) = \exp(-d_{\psi^*}(b;\theta)/\sigma^2)p_0(b,\sigma^2)$ is a log-concave density function for all $b \in {\cal B}$ even $\dom\psi^*$ is not open (see Theorem \ref{legendreth} (2)).
Now, we only need to show the existence of a function between the mean parameter space and the canonical parameter space in an extended-valued real number system. Let us consider \eqref{optx1}:
$ 
\nabla\psi^*(\E(b)) = \argmin_{\theta \in int(dom \psi)} d_{\psi^*}(\E(b);\theta). 
$
From Theorem \ref{canTH}, when $\E(b) \in \Omega_1$, we have $\nabla\psi^*(\E(b)) \in \{\pm\infty\}$. It means that, due to the steepness condition of Legendre, there is a sequence $\{ \mu_m\} \in int(B)$ such that $\mu_m \rightarrow \E(b) \in \Omega_1$ as $m \rightarrow \infty$. Therefore, we have 
$\nabla\psi^*(\mu_m) = \theta_m \rightarrow \{\pm\infty\}$ and thus $\nabla\psi^*: int(B)\cup\Omega_1 \rightarrow \dom\psi \cup \{ \pm\infty \}$ exists  within an extended-valued real number system for all $\E(b) \in \dom\psi^* = int(B)\cup\Omega_1$. In consequence, $\psi$ is a cumulant function of the log-concave density function $p_{\psi}$ having the first cumulant, i.e., $p_{\psi}$ is in the $1$-LED model. 

(b) Consider $d_{\psi}(\theta;b) = \psi(\theta) + \psi^*(b) - \inprod{\theta}{b}$ with $\dom\psi (\not= int(\dom\psi))$ and $\dom\psi^* =  int(\dom\psi^*)$. In Theorem \ref{openeq}, the case $\theta \in int(\dom\psi)$ is treated. Let $\theta \in \Omega_2 = \dom\psi \setminus int(\dom\psi)$. As done in (a), there is a sequence $\{ \theta_m\} \in int(\dom\psi)$ with $\nabla\psi(\theta_m) = \mu_m \in int(\dom\psi^*)$ such that $\theta_m \rightarrow \theta_{\infty} = \theta \in \Omega_2$ satisfying
$
\lim_{m \rightarrow \infty}  \nabla\psi(\theta_m) = \mu_{m} \in \{\pm\infty\}.
$
Therefore, we have $\nabla\psi(\theta) \in \{\pm\infty\}$ where $\theta \in \Omega_2$ and $\E(b) \in \{\pm\infty\}$. That is, $\nabla\psi : int(\dom\psi) \cup \Omega_2 \rightarrow \dom\psi^* \cup \{ \pm \infty \}$ in an extended-valued real number system. However, from \cite[Th. 14.17]{bauschke11}, when $\theta \in \Omega_2$, $d_{\psi}(\theta;b)$ is not coercive in terms of $b$. Therefore, $p_{\psi}$ is a log-concave density function and a non-full $1$-LED model only when $\theta \in int(\dom\psi)$.
\end{proof}

\begin{example}
Consider Bernoulli distribution. Let $b_i \in {\cal B} = \{0,1\}$ and $\E(b) = \mu \in B = [0,1]$, then we have Bregman divergence associated with Legendre $\psi$:
$
d_{\psi}(b;\theta) = D_{\psi}(b|\mu) = b \ln\left(\frac{b}{\mu}\right) + (1-b)\ln\left(\frac{1-b}{1-\mu}\right)
$
where $\psi^*(\theta) = \log(1 + \exp(\theta))$ and $\psi(\mu) = \mu\ln\mu + (1-\mu)\ln(1-\mu)$. Thus, we have Bregman-divergence-guided log-concave density function
$
p(b;\theta,\sigma^2) = \exp(-d_{\psi}(b;\theta)/\sigma^2)p_0(b,\sigma^2)
$
where $p_0(b,\sigma^2) = \exp(b\ln b + (1-b)\ln(1-b))/\sigma^2$. Thus, the corresponding maximum likelihood estimation is
$$
\max_{\theta}\prod_{i=1}^M p(b_i;\theta,1) = p(b)\exp\left(\sum_{i=1}^M b_i\log\mu + (1-b_i)\log(1-\mu)\right) = p(b)\exp\left(\inprod{\sum_{i=1}^Mb_i}{\theta} - M\psi^*(\theta)\right).
$$ 
where $p(b) = \prod_{i=1}^M p_0(b_i,1)$.
Then
$
\mu_{avg}= \frac{1}{M}\sum_{i=1}^M b_i = \nabla\psi^*(\theta_{avg}) = \hbox{sigm}(\theta_{avg})
$
where $\hbox{sigm}(\theta) = \frac{1}{1 + \exp(-\theta)}$ is a sigmoid function. Within an extended-valued real number system, we have
$$
\nabla\psi: (0,1) \cup \{0,1 \} \rightarrow \R \cup \{\pm\infty\} 
$$
where $\nabla\psi(0) = -\infty$, $\nabla\psi(1) = +\infty$. 
\end{example}
Note that, in Section \ref{berpoi}, we introduce the parameterized K-LED model connecting Bernoulli distribution and Poisson distribution through the convex extended logistic loss function~\cite{woo19}.

\begin{example}\label{exIG}
Consider inverse Gaussian distribution~\cite{jorgensen97}.
This distribution is a typical non-regular $1$-LED model on $\dom\Psi$ (or non-full $1$-LED model on $int(\dom\Psi)$). Let us consider Bregman-Tweedie divergence:
$
d_{\Psi}(\theta;b) = \Psi(\theta) + \Phi(b) - \inprod{\theta}{b}.
$
From Appendix, when $\beta=-1$, we have $\Phi(b) = (2b)^{-1}$ with $\dom\Phi = \R_{++}$ and $\Psi(\theta) = -\sqrt{-2\theta}$ with $\dom\Psi = \R_-$. Note that the canonical parameter space $\dom\Psi$ is not open. As observed in \cite{jorgensen97}, we have 
$
p(b;\theta,\sigma^2) = p_0(b,\sigma^2)\exp\left(-d_{\Psi}(\theta;b)/\sigma^2\right),
$ where $p_0(b,\sigma^2) = \frac{1}{\sqrt{2\pi\sigma^2 b^3}}$. If $\theta \in \dom\Psi \setminus int(\dom\Psi)$, i.e., $\theta=0$, then $\Phi(b) - \inprod{\theta}{b} = \frac{1}{2b}$ is not coercive. Thus, we have
\begin{equation}\label{levy}
p(b;0,\sigma^2) = \frac{1}{\sqrt{2\pi\sigma^2b^3}}\exp(-(\sigma^{2}2b)^{-1})\qquad \left( \approx \quad \frac{a(\sigma)}{b^{1+\alpha}} \quad \hbox{as} \;\; b \rightarrow \infty \right)
\end{equation}
This is Levy distribution (or a stable distribution with index $\alpha=0.5$) which does not have a mean value. 
\end{example}
Interestingly, Tweedie distribution between gamma distribution and inverse Gaussian distribution, i.e., $d_{\Psi}(\theta;b)$ with $-1 \le \beta \le 0$, is non-regular but it is useful for segmenting the SAR(synthetic aperture radar) images. For more details, see \cite{muller02,woob16,woo17} and Section \ref{sec4}.

\section{The $2$-LED model with mean-variance relation\label{secQ}}
This Section introduces the $2$-LED model~\eqref{EDx} with the mean-variance relation $var(b) = \sigma^2 V(\mu)$ of quasi-likelihood function~\eqref{qlike}.  For simplicity, we only consider ${\cal B} \subseteq \R^1$ with  a random variable $b \in {\cal B}$.

\subsection{The $2$-LED model for quasi-likelihood function\label{secQA}}
From \cite{mccullagh89}, a quasi-score function $\frac{\partial Q(b;\mu)}{\partial \mu} = \frac{b-\mu}{\sigma^2 V(\mu)}$ fulfills the  properties of a score function (the first derivative of a log likelihood function):
$
\begin{array}{l}
\E\left( \frac{\partial Q(b;\mu)}{\partial \mu} \right) = 0$ and 
$\E\left(\left(\frac{\partial Q(b;\mu)}{\partial \mu}\right)^2 \right) = {\cal I}_{\mu}
\end{array}
$
where $\E[b] = \mu$ and ${\cal I}_{\mu} = 1/var(b)$ is the Fisher information. From this, we get the mean-variance relation of quasi-likelihood function:
\begin{equation}\label{mvr}
var(b) = \sigma^2 V(\mu) \ge 0
\end{equation}
where $\sigma^2>0$ is a constant and $V(\mu)$ is a unit variance function~\cite{wedderburn74}. As observed in \cite{woob16}, quasi-likelihood function $Q(b;\mu)$ satisfies the global optimum property, i.e.,
$\E(Q(b;u)) \le Q(b;\E(b))$ for all $u$ on its domain. See also Theorem \ref{legendreth} (6) for Bregman divergence. Since quasi-likelihood function only requires the existence of mean and variance, it includes relatively large extended probability distribution including classical distributions. For instance, when $V(\mu) = \mu(1-\mu)$, the corresponding quasi-likelihood function is the log-likelihood function of Bernoulli distribution. However, if $V(\mu) = \mu^2(1-\mu)^2$ then we have $Q(b;\mu) = (2b-1)\log\left(\frac{\mu}{1-\mu}\right) - \frac{b}{\mu} - \frac{1-b}{1-\mu}$. As noticed in~\cite{mccullagh89}, this function does not have a corresponding cumulant function. Another interesting log-likelihood of the extended probability distribution is quasi-likelihood function with power variance function $V(\mu) = \mu^{2-\beta}$~\eqref{powerlaw}. As noticed in Section \ref{betasec}, this is a subclass of $\beta$-divergence and tightly connected to the regular $2$-LED model with \eqref{powerlaw}.
 
In fact, mean-variance relation \eqref{mvr} can be formulated with the first and the second cumulant of the K-LED model~\eqref{EDx}:
\begin{equation}\label{mvr2}
\nabla^2 \psi(\theta) = V(\nabla \psi(\theta))
\end{equation}
where $var(b) = \sigma^2 \nabla^2 \psi(\theta)$. See also \cite{kass97} for related work. For Poisson distribution, we have $\nabla^2\psi(\theta) = \nabla \psi(\theta)$ and thus $\psi(\theta) = \exp(\theta)$. But, for $V(\mu) = \mu^2(1-\mu)^2$, we do not have a closed-form solution $\psi$ satisfying \eqref{mvr2}. In case of Tweedie distribution, it is rather complicated to find $\psi$ satisfying $\nabla^2\psi(\theta) = (\nabla\psi(\theta))^{\beta-2}$ for all $\beta \in \R$. This case is separately treated in Section \ref{sec4}. First of all, let us introduce the connection between quasi-likelihood function and Bregman divergence associated with Legendre. 
\begin{theorem}\label{2led}
Let $b \in {\cal B} \subseteq \dom\psi^*$ with \eqref{assumeKLED}, $\psi \in \L_1 \cap C^2(int(\dom\psi))$, and $\psi^* \in \L_1 \cap C^2(int(\dom\psi^*))$. In addition, let $\nabla\psi(\theta) = \mu \in int(\dom\psi^*)$ be the first cumulant. Under the regularity condition $\nabla^2\psi(\theta) = V(\nabla\psi(\theta)) > 0$ for all $\theta \in int(\dom\psi)$, we have
\begin{equation}\label{taylorBreg}
d_{\psi^*}(b;\theta) = D_{\psi^*}(b|\mu) = - \sigma^2Q(b;\mu)
\end{equation}
Therefore, when $\dom\psi = int(\dom\psi)$, the log-concave density function $p_{Q}(b;\mu,\sigma^2) = \exp(Q(b;\mu))p_0(b,\sigma^2)$ is a regular $2$-LED model with mean-variance relation~\eqref{mvr2}. Here, $p_0(b,\sigma^2)$ is a base measure satisfying $\int_{\cal B} p_Q(b;\mu,\sigma^2)\nu(db) = 1$.
\end{theorem}
\begin{proof}
Since $\psi \in \L_1 \cap C^2 (int(\dom\psi))$, we have $\nabla\psi \circ \nabla\psi^*(\mu) = \mu$ for all $\mu \in int(\dom\psi^*)$. Thus $\nabla^2\psi(\theta) \nabla^2\psi^*(\mu) = 1$ where $\theta = \nabla\psi^*(\mu) \in int(\dom\psi)$. That is to say, we have 
\begin{equation}\label{2ndOrd}
V(\mu)\nabla^2\psi^*(\mu)=1
\end{equation} 
Consider the Taylor approximation of $\psi^*(b)$ with $b \in {\cal B} \subseteq \dom\psi^*$ at $\mu \in int(\dom\psi^*)$:
$
\psi^*(b) = \psi^*(\mu) + \nabla\psi^*(\mu)(b-\mu) + \int_{\mu}^b \nabla^2\psi^*(y)(b-y)dy.
$
From \eqref{2ndOrd}, we have
\begin{equation*}
D_{\psi^*}(b|\mu) = \int_\mu^b \nabla^2\psi^*(y) (b-y) dy = \int_\mu^b \frac{(b-y)}{V(y)} dy = - \sigma^2 Q(b;\mu)
\end{equation*}
By Theorem \ref{nopeneq1} (a), it is easy to see that $p_{Q}(b;\mu,\sigma^2) = \exp(Q(b;\mu))p_0(b,\sigma^2)$ is a log-concave density function with a base measure $p_0(b,\sigma^2)$ satisfying $\int_{\cal B} p_Q(b;\mu,\sigma^2)\nu(db) =1$. When $\dom\psi=int(\dom\psi)$, it becomes the regular $2$-LED model $p_Q(b;\theta,\sigma^2) = \exp(-d_{\psi^*}(b;\theta)/\sigma^2)p_0(b,\sigma^2)$ where $d_{\psi^*}(b;\theta) = D_{\psi^*}(b|\nabla\psi(\theta))$. 
\end{proof}
Theorem \ref{2led} shows the equivalence between  quasi-likelihood function and the regular $2$-LED model under the regularity condition \eqref{mvr2}, $\dom\psi$ is open, and $var(b)>0$. However, variance does not need to be always positive, even though cumulant function is strictly convex. For instance, $\psi(\theta) = \frac{3}{4}\left(\frac{1}{3}\theta\right)^4$ is the cumulant function of a degenerate Tweedie distribution ($\beta = 4/3$) with $\nabla\psi(0) = \nabla^2\psi(0)=0$. Actually, if $var(b) = \sigma^2\nabla^2\psi(\theta) \ge 0$ then the regularity condition $\psi^* \in \L_1 \cap C^2(int(\dom\psi^*))$ should be relaxed. In other words, $\psi^* \in \L_1 \cap C^1(int(\dom\psi^*))$ and thus the relation \eqref{2ndOrd} is not available anymore. The details are following.
\begin{lemma}\label{QuasiTweedie}
Consider a quasi-likelihood function $Q:\Omega_L \times \Omega_R \rightarrow \R_-$ with $V(\mu) = \mu^{2-\beta}$ and $var(b) = \sigma^2V(\mu)$. This is a scaled $\beta$-divergence $Q(b;\mu) = -D_{\beta}(b|\mu)/\sigma^2$. Here $\sigma^2>0$ is a constant. Then, $\dom V(\mu) = \{ \mu \in \R \;|\; V(\mu) \ge 0 \}$ becomes 
\begin{eqnarray}\label{domV}
\dom V(\mu) = 
\left\{ 
\begin{array}{l}
\R, \;\;\;\;\quad \beta \le 2,\; \beta \in \R_e\\
\R_+, \;\;\quad \beta \le 2\\ 
\R_{--}, \quad \beta > 2,\; \beta \in \R_e\\
\R_{++}, \quad \beta > 2 
\end{array}
\right.
\end{eqnarray}
and the corresponding (effective) domain $\Omega_L \times \Omega_R$ is summarized in Table \ref{table10}.  
\end{lemma}
\begin{proof}
It is trivial to compute \eqref{domV}. By the intersection of \eqref{domV} and the domain of $\beta$-divergence $D_{\beta}(b|\mu)$ in \cite{woo17}, we get the domain of a quasi-likelihood function with~\eqref{powerlaw} in Table \ref{table10}. 
\end{proof}
When $\beta > 2$, the domain of $Q(b;\mu)$ does not exist, unlike $\beta$-divergence in \cite{woo17}. However, in case $\beta \in (1,2) \cap \R_e$, $Q(b;\mu)$ with \eqref{powerlaw} exists and it has $\dom Q(b;\mu) = \R \times \R$. See \cite{eguchi01,samek13,woo19,woo19b} for applications of this region $\beta \in (1,2)$ in data classification with stabilized high-order hinge loss and  robust generalized distance between two probability distribution, i.e., robust KL-divergence. Now, we present Bregman-beta-guided log-concave density function having power variance function~\eqref{powerlaw} (see also Example \ref{TweedieX} and \cite[Page 336]{mccullagh89}). 
\begin{theorem}[regular $2$-LED for Tweedie]\label{2LEDr}
Let $b \in {\cal B} \subseteq \dom\Phi$, $\Phi \in \L_1  \cap C^1(int(\dom\Phi))$, and $\Psi \in \L_1 \cap C^2(int(\dom\Psi))$ with $\nabla\Psi(\theta) = \E[b]$. Consider Bregman-beta divergence $d_{\Phi}(b;\theta) = \Phi(b) + \Psi(\theta) - \inprod{b}{\theta}$. If $\beta \in [0,1]$ or $\beta \in (1,2] \cap \R_e$ then $p_{\Psi}(b;\theta,\sigma^2) = \exp(-d_{\Phi}(b;\theta)/\sigma^2)p_0(b,\sigma^2)$ is a log-concave density function with $\dom\Psi = int(\dom\Psi)$ and $\nabla^2\Psi(\theta)=V(\nabla\Psi(\theta)) \ge 0$. Here $\sigma^2>0$ is a constant and $p_0(b,\sigma^2)$ is a base measure satisfying $\int_{\cal B} p_{\Psi}(b;\theta,\sigma^2)\nu(db) =1$. The corresponding regular $2$-LED model with \eqref{powerlaw} is
$p_{\Psi}(b;\theta,\sigma^2) = p_1(b,\sigma^2) \exp\left(\frac{\inprod{b}{\theta} - \Psi(\theta)}{\sigma^2}\right)$
where $p_1(b,\sigma^2) = p_0(b,\sigma^2) \exp(-\Phi(b)/\sigma^{2})$. Additionally, $\dom\Psi$ and $\dom\Phi$ are summarized in Table \ref{table10}.
\end{theorem}
\begin{proof}
Consider $\Psi$ with $\dom\Psi = int(\dom\Psi)$ and $\beta \ge 0$ (see \eqref{conditionLegendreD}). As noticed in Theorem \ref{legendreth} (2), $\forall\theta \in \dom\Psi$, $d_{\Phi}(b;\theta)$ is coercive with respect to $b \in \dom\Phi$. Thus $p_{\Psi}(b;\theta,\sigma^2) = \exp(-d_{\Phi}(b;\theta)/\sigma^2)p_0(b,\sigma^2)$ is a log-concave density function. Here, $p_0(b,\sigma^2)$ is a base measure satisfying $\int_{\cal B}p_{\Psi}(b;\theta,\sigma^2)\nu(db)=1$. Note that $\nabla^2\Psi(\theta) = [(\beta-1)\theta]^{\frac{2-\beta}{\beta-1}} = V(\nabla\Psi(\theta)) = (\exp_{2-\beta}(\theta))^{2-\beta}$. Then, when $\beta \le 2$, we have $V(\nabla\Psi(\theta)) \ge 0$. Actually, it corresponds to the case $\beta \in [0,1]$ or $\beta \in (0,2]\cap \R_e$. Last but not least, since $\nabla^2\Phi(b) = b^{\beta-2}$, when $\beta \in (1,2) \cap \R_e$, $\Phi \not\in C^2(int(\dom\Phi))$.  See Theorem \ref{corPsi} and Theorem \ref{legendre_beta} for more details.
\end{proof}
From Table \ref{table10}, we see that $\dom\Phi$ is non-open when $\beta \in (0,1]$. This corresponds to the compound Poisson-gamma distribution ($\beta \in (0,1)$) and Poisson distribution ($\beta = 1$) of the classic Tweedie distribution. 

\begin{remark}[non-regular $2$-LED for Tweedie]\label{2LEDnr}
Let $b \in {\cal B} \subseteq \dom\Phi$, $\Phi \in \L_1  \cap C^2(int(\dom\Phi))$, and $\Psi \in \L_1 \cap C^2(int(\dom\Psi))$ with $\nabla\Psi(\theta) = \E[b]$. Consider Bregman-Tweedie divergence $d_{\Psi}(\theta;b) = \Psi(\theta) + \Phi(b) - \inprod{b}{\theta}$ and $\sigma^2>0$ is a constant. If $\beta < 0$ then, by simple calculation, we get $\nabla^2\Psi(\theta)=V(\nabla\Psi(\theta)) > 0$ with $\dom\Psi = cl(\dom\Psi)$ and $\dom\Phi = int(\dom\Phi)$. From Theorem \ref{nopeneq1} (b), when $\theta \in int(\dom\Psi)$, $p_{\Psi}(b;\theta,\sigma^2) = \exp\left(-d_{\Psi}(\theta;b)/\sigma^2\right)p_0(b,\sigma^2) = \exp(\inprod{b}{\theta} - \Psi(\theta))p_1^{nr}(b,\sigma^2)$ is a non-full 2-LED model. Here $p_1^{nr}(b,\sigma^2) = p_0^{nr}(b,\sigma^2)\exp(-\Phi(b)/\sigma^2)$ is a base measure satisfying $\int_{\cal B} p_{\Psi}(b;\theta,\sigma^2)\nu(db)=1$. However, when $\theta \in \dom\Psi \setminus int(\dom\Psi)$ (i.e., $\theta=0$), $p_{\Psi}(b;0,\sigma^2) =  p_1^{nr}(b,\sigma^2)$ is not a log-concave density function.  In fact, $\lim_{\abs{b} \rightarrow \infty} \exp(-\Phi(b)/\sigma^2) = 1$. Therefore, to satisfy $\int_{\cal B} p_{\Psi}(b;0,\sigma^2) \nu(db) = 1$, $\lim_{\abs{b} \rightarrow \infty} p_0^{nr}(b,\sigma^2)=0$ is required. For instance, when $\beta=-1$ (inverse Gaussian distribution), we get $p_0^{nr}(b,\sigma^2) = (2\pi \sigma^2 b^3)^{-1/2}$ (see Example \ref{exIG}). Note that the case $\beta \in (0,-1) \cup (-1,-\infty)$ is known as the positive stable distribution. For more details, see \cite{jorgensen97}. Last but not least, quasi-probability distribution $p_Q(b;\mu,\sigma^2) = \exp(Q(b;\mu))p_0(b,\sigma^2)$ with $\beta<0$ and \eqref{powerlaw} is a non-full $2$-LED model. See Table \ref{table10} and Theorem \ref{2led} for more details.
\end{remark}

 \begin{table*}
\centering
\begin{tabular}{l|ll||l|lll}
\hline\hline
\textbf{Quasi-likelihood with \eqref{powerlaw}} & &&\textbf{Bregman-beta} (or \textbf{-Tweedie}) & & &  \\ \hline
 $1 < \beta \le 2$ & $\Omega_L=\R_+$ & $\Omega_R = \R_+$ & - & - & - \\ 
$1 < \beta \le 2, \beta \in \R_e$ & $\Omega_L=\R$ & $\Omega_R = \R$ & $1 < \beta \le 2, \beta \in \R_e$ & $\dom\Phi = \R$ & $\dom\Psi = \R$ \\ 
- & - & - & $\beta=1$ & $\dom\Phi = \R_+$ & $\dom\Psi = \R$ \\  
 $0 < \beta \le 1$ & $\Omega_L=\R_+$ & $\Omega_R = \R_{++}$ & $0 < \beta < 1$ & $\dom\Phi = \R_+$ & $\dom\Psi = \R_{--}$ \\ 
  $0<\beta<1, \beta \in \R_e$ & $\Omega_L = \R_-$ & $\Omega_R = \R_{--}$ & $0 < \beta < 1, \beta \in \R_e$ & $\dom\Phi = \R_-$ & $\dom\Psi = \R_{++}$ \\
 - & - & - & $\beta = 0$ & $\dom\Phi = \R_{++}$ & $\dom\Psi = \R_{--}$  \\ \hline    
$\beta < 0$ & $\Omega_L = \R_{++}$ &$\Omega_R = \R_{++}$ & $\beta < 0$ & $\dom\Phi = \R_{++}$ & $\dom\Psi = \R_{-}$  \\  
${ \beta  \le  0}, \beta \in \R_e$ & $\Omega_L = \R_{--}$ & $\Omega_R = \R_{--}$ & ${\beta <  0}, \beta \in \R_e$ & $\dom\Phi = \R_{--}$ & $\dom\Psi = \R_{+}$  \\ \hline\hline
  \end{tabular}
\caption{A comparison of the domain of a quasi-likelihood function with power variance function~\eqref{powerlaw} and that of  Bregman-beta divergence ($\beta \ge 0$) (or of Bregman-Tweedie divergence ($\beta <0$)) with $V(\mu) \ge 0$. It is worth mentioning about the region $\beta<0$. In this region,
 quasi-likelihood function with~\eqref{powerlaw} corresponds to the non-full $2$-LED model. However, the $2$-LED model, which is based on  Bregman-Tweedie divergence, is a non-regular exponential dispersion model. 
\label{table10}
}
\end{table*}
   
\subsection{The K-LED model with power variance function~\eqref{powerlaw}\label{sec4}}
This Section explores the K-LED model with power variance function~\eqref{powerlaw}. The K-LED model with \eqref{powerlaw} and analytic cumulant function becomes the classic Tweedie exponential dispersion model in an extended domain. See also~\cite{bar-lev86}.

\begin{table*}[h]
\centerline{\begin{tabular}{l||lcl|c|cc}
\hline\hline
 & &  $\bm \beta \bm > \bm 1$ ($\beta \in \R_e$)  & & $\bm \beta \bm = \bm 1$ & $\bm \beta \bm < \bm 1$ ($\beta \in \R_e$) & $\bm \beta \bm < \bm 1$ ($\beta \not\in \R_e$) \\ \hline\hline 
 $\dom (\nabla \Psi(\theta))$ &  &   & $\R$ ($\beta>1$)  & $\R$ & $\R_{++}/\R_{--}$ & $\R_{--}$  \\ \hline
 $\dom(\nabla^{2}\Psi(\theta))$ & $\R_{++} / \R_{--}$ ($\beta>2$) & $\R$ ($\beta = 2$) & $\R$ ($2>\beta>1$)  & $\R$ & $\R_{++}$/$\R_{--}$ & $\R_{--}$ \\ \hline
 $\dom(\nabla^{3}\Psi(\theta))$ & $\R_{++} / \R_{--}$ ($\beta>\frac{3}{2}$) & - & $\R$ ($\frac{3}{2}>\beta>1$)  & $\R$ & $\R_{++}$/$\R_{--}$ & $\R_{--}$ \\ \hline
 $\dom(\nabla^{4}\Psi(\theta))$ & $\R_{++} / \R_{--}$ ($\beta>\frac{4}{3}$) & $\R$ ($\beta=\frac{4}{3}$) & $\R$ ($\frac{4}{3}>\beta>1$)  & $\R$ & $\R_{++}$/$\R_{--}$ & $\R_{--}$ \\ \hline 
  $\dom(\nabla^{5}\Psi(\theta))$ & $\R_{++} / \R_{--}$ ($\beta>\frac{5}{4}$) & - & $\R$ ($\frac{5}{4}>\beta>1$)  & $\R$ & $\R_{++}$/$\R_{--}$ & $\R_{--}$ \\ \hline \hline
$\dom(\nabla^{2k}\Psi(\theta))$ & $\R_{++} / \R_{--}$ ($\beta>\frac{2k}{2k-1}$)  & $\R$ ($\beta=\frac{2k}{2k-1}$) & $\R$ ($\frac{2k}{2k-1}>\beta>1$)  & $\R$ & $\R_{++}$/$\R_{--}$ & $\R_{--}$ \\ \hline
$\dom(\nabla^{2k+1}\Psi(\theta))$ & $\R_{++} / \R_{--}$	 ($\beta>\frac{2k+1}{2k}$)  & - & $\R$ ($\frac{2k+1}{2k}>\beta>1$)  & $\R$ & $\R_{++}$/$\R_{--}$ & $\R_{--}$ \\ \hline
\end{tabular}}
\caption{The domain of $k$-th cumulant function of $\Psi$~\eqref{basefnD} depending on selection of $k$. In particular, when $\beta= \left(1+\frac{1}{k-1}\right) \cap \R_e$, $\nabla^{k}\Psi(\theta) = \eta(k,\beta)$ is a constant and thus the domain of it is simply $\R$. Here, $\eta(k,\beta) = 1(2-\beta)(3-2\beta)...(k-1 -(k-2)\beta).$ As we increase $k$, the region $1<\beta<1+\frac{1}{k-1}$ becomes empty set and thus the classic Tweedie distribution does not exists when $\beta > 1$, except $\beta=2$. See Theorem \ref{kthcu} for more details. 
}\label{table4}
\end{table*}

\begin{table*}[h]
\centerline{\begin{tabular}{c|l|c||cc||ccc||c}
\hline\hline
    &               &                   &                        & {\bf $1$-LED} with \eqref{powerlaw}                    &  &  &  & {\bf Tweedie}\\ \cline{2-9}
{\bf Distribution} & $\beta$  & ${\cal B}$  & $\dom \Phi$ & $\dom \Psi$ & $\dom \nabla\Psi$ & $\dom(\nabla^2\Psi)$ & $\dom(\nabla^{K}\Psi)$ & $\dom\Psi$  \\ \hline\hline
-$^{\bm 1}$ & ${ 2 < \beta}$, $\beta \in \R_e$   & $\R$ & $\R$ & $\R$ & $\R$ & - & - & - \\  \hline 
Gaussian & ${ \beta  =  2}$   & $\R$ & $\R$ & $\R$ & $\R$ & $\R$ & - & $\R$ \\  \hline 
-$^{\bm 2}$ & ${\frac{K}{K-1} <  \beta < 2}$, $\beta \in \R_e$   & $\R$ & $\R$ & $\R$ & $\R$ & $\R$ & - & - \\  \hline  
-$^{\bm 2}$ & ${\frac{K}{K-1} =  \beta}$, $\beta \in \R_e$   & $\R$ & $\R$ & $\R$ & $\R$ & $\R$ & $\R$ & - \\  \hline
-$^{\bm 2}$ & ${1< \beta < \frac{K}{K-1}}$, $\beta \in \R_e$   & $\R$ & $\R$ & $\R$ & $\R$ & $\R$ & $\R$ & - \\  \hline
Poisson  & $\beta =  1$ & $\mathbb{Z}_+$ & $\R_{+}$& $\R$ & $\R$ & $\R$ & $\R$ & $\R$ \\  \hline 
Compound  & $ 0 < \beta <  1$, $\beta \not\in \R_e$ & $\R_+$  & $\R_+$ & $\R_{--}$ & $\R_{--}$ & $\R_{--}$ & $\R_{--}$ & $\R_{--}$ \\
Poisson-Gamma & $ 0  <  \beta <  1$, $\beta \in \R_e$ & $\R_+$ & $\R_+/\R_-$ & $\R_{--}/\R_{++}$ & $\R_{--}/\R_{++}$ & $\R_{--}/\R_{++}$ & $\R_{--}/\R_{++}$ & $\R_{--}/\R_{++}$ \\  \hline
Gamma  & $\beta = 0$ & $\R_{++}$  & $\R_{++}$ & $\R_{--}$ & $\R_{--}$ & $\R_{--}$ & $\R_{--}$ & $\R_{--}$ \\ \hline
Inverse Gaussian  & $ \beta = {-1}$ & $\R_{++}$ & $\R_{++}$ & $\R_{-}$ & $\R_{--}$ & $\R_{--}$ & $\R_{--}$ & $\R_{-}$ \\ \hline
Positive stable  & ${\beta < 0}$, ${\beta \not= -1}$, $\beta \not\in \R_e$ & $\R_{++}$ & $\R_{++}$ & $\R_{-}$ & $\R_{--}$ & $\R_{--}$ & $\R_{--}$ & $\R_{-}$ \\
 & ${ \beta < 0}$, ${ \beta \not= -1}$, $\beta \in \R_e$ & $\R_{++}$ & $\R_{++}/\R_{--}$ & $\R_{-}/\R_{+}$ & $\R_{--}/\R_{++}$ & $\R_{--}/\R_{++}$ & $\R_{--}/\R_{++}$ & $\R_{-}/\R_{+}$ \\ \hline 
\end{tabular}}
\caption{A characterization of the K-LED model with \eqref{powerlaw}. $^{\bm 1}$($K=1$): It is denoted as the $1$-LED model with \eqref{powerlaw}. This is induced from Bregman-beta divergence~\eqref{betaBregman}. See Example \ref{TweedieX} for more details. $^{\bm 2}$($2 < K < \infty$): These are not classic probability distributions having infinitely many cumulant functions. However, when $\beta \in (1,\frac{K}{K-1}]$, it is degenerate K-LED models with \eqref{powerlaw}. In fact, we have $\nabla\Psi(0) = \nabla^2\Psi(0) = 0$. Notice that, when $\beta \in [0,1]$, the classic Tweedie distribution is derived from Bregman-beta divergence~\eqref{betaBregman} and, when $\beta<0$, the classic non-regular Tweedie distribution is derived from Bregman-Tweedie-divergence~\eqref{tweedieBregman}. 
}\label{table5}
\end{table*}
\begin{theorem}\label{kthcu}
Let $K \ge 2$, $b \in {\cal B} \subseteq \dom \Psi^*$ be a random variable (or an observation), and  $\Psi \in \L_1 \cap C^K(int(\dom\Psi))$ where $\dom\Psi$ depends on $\beta$ and $\Psi$~\eqref{basefnD}. Let $\beta \in (-\infty,1] \cup \{(1,2]\cap \R_e\}$, then we have K-LED with \eqref{powerlaw}:
\begin{equation}\label{pX}
p_{\Psi}(b;\theta,\sigma^2) = \exp\left(\frac{\inprod{b}{\theta}-\Psi(\theta)}{\sigma^2}\right)p_1(b,\sigma^2)
\end{equation}
where $\sigma^2>0$ is a dispersion parameter and  $p_1(b,\sigma^2) = p_0(b,\sigma^2)\exp(-\Phi(b)/\sigma^2)$ is a base measure satisfying $\int_{{\cal B}} p_{\Psi}(b;\theta,\sigma^2)\nu(db) = 1$. The cumulant function $\Psi$ is analytic on
\begin{equation}\label{domTW}
int(\dom\Psi) =   
\left\{
\begin{array}{l}
\R, \qquad\quad\quad\;\; \beta = 1, 2 \\ 
\R_{++} / \R_{--}, \;\;\;\; \beta<1, \; \beta \in \R_e \setminus \{ 0 \}\\
\R_{--}, \qquad\quad\;\; \beta = 0, \;\beta<1, \; \beta \not\in \R_e\\
\end{array}
\right.
\end{equation}
\end{theorem}
\begin{proof}
From Theorem \ref{2LEDr} and Remark \ref{2LEDnr}, when $\beta \in (-\infty,1]$ and $\beta \in (1,2] \cap \R_e$, the $2$-LED model with~\eqref{pX} and $\dom\Psi$ in Table \ref{table10}, satisfies  power variance function~\eqref{powerlaw}. Consider $\nabla\Psi(\theta) = \exp_{2-\beta}(\theta) = \mu$ where $\theta \in int(\dom\Psi) = \dom(\nabla\Psi)$. For completeness, we summarize $\dom(\nabla\Psi)$ in Table \ref{table4}.  Assume that $k \ge 2$ and the  $k$-th cumulant of  K-LED with \eqref{powerlaw} in \eqref{pX}:
\begin{equation}\label{nthC}
\nabla^{k} \Psi(\theta) =  \eta(k,\beta)[(\beta-1)\theta]^{H(\beta,k)} 
\end{equation}
where $\beta \not= 1$, $\eta(k,\beta) = 1(2-\beta)(3-2\beta)...(k-1 -(k-2)\beta)$, and $H(\beta,k) = \frac{1}{\beta-1}-(k-1)$. The following classification is useful for deciding the domain of $k$-th cumulant with $\beta \not= 1$:
\begin{itemize}
\item $\beta \in \R_e$ and $k$ is even:\;
Since $\frac{1}{\beta-1} \in \R_o$ and $k-1$ is odd, 
$
H(\beta,k)  \in \R_e.
$
\item $\beta \in \R_e$ and $k$ is odd:\;
Since $\frac{1}{\beta-1} \in \R_o$ and $k-1$ is even, 
$
H(\beta,k)  \in \R_o.
$
\item $\beta \not\in \R_e$: 
$
H(\beta,k) \in \R_x.
$ 
\end{itemize}
Now, the domain of $k$-th cumulant \eqref{nthC} ($k \ge 2$) is classified as
\begin{itemize}
\item $\beta = 1$: Since $\Psi(\theta)=\exp(\theta)$,  we have $\dom(\nabla^{k}\Psi(\theta)) = \R$, irrespective of $k$.
\item $\beta<1$: $H(\beta,k)<0$, irrespective of the choice of $k$, and thus we have $\dom (\nabla^{k}\Psi(\theta)) = \R_{--}$. If $\beta \in \R_e \setminus \{ 0 \}$ then $\dom (\nabla^{k}\Psi(\theta)) = \R_{++}$ is also possible.
\item $\beta \in (1,1+\frac{1}{k-1})$ and $k \ge 2$: If $\beta \in \R_e$ then $H(\beta,k) > 0$ and $H(\beta,k) \in \R \setminus \R_x$. Therefore, $\dom(\nabla^{k}\Psi(\theta)) = \R.$
\item $\beta > 1 + \frac{1}{k-1}$ and $k \ge 2$: $H(\beta,k)<0$ and $\beta-1>0$. Therefore, $\dom(\nabla^k \Psi(\theta)) = \R_{++}$. If $\beta \in \R_e$ then $H(\beta,k) \in \R \setminus \R_x$ and $\dom(\nabla^{k}\Psi(\theta)) = \R_{--}$ is also possible. However, this case does not preserve $\dom\Psi = \R$.
\item $\beta = \left(1 + \frac{1}{k-1}\right) \cap \R_e$ and $k \ge 2$: $H(\beta,k)=0$ and thus 
$
\nabla^{k}\Psi(\theta) = \eta(k,\beta).
$ 
That is, $\nabla^{k}\Psi(\theta)$ does not depend on parameter $\theta$. Hence, $\dom(\nabla^{k}\Psi(\theta)) = \R.$  For all $m>k$, if we set $\beta=1+ 1/(k-1)$ then $\eta(m,\beta) = 0$ and thus
$
\nabla^{m}\Psi(\theta) = 0. 
$ That is, the higher moments ($\forall m>k$) do not exist. In fact, when $k=2$, we have $\beta=2$ and thus $\nabla^{2}\Psi(\theta) = 1$. It corresponds to the normal distribution. Unfortunately, when $2<k<+\infty$, we do not have corresponding probability distributions, due to the Marcinkievicz Theorem~\cite{lukacs58}. Note that when $\beta = \left(1 + \frac{1}{k-1}\right) \cap (\R \setminus \R_e)$, $\Psi$ is not Legendre.
\end{itemize}
Then, we have the domain of the classic Tweedie distribution, described in \eqref{domTW}. 
\end{proof}
We summarize the domains of K-LED in Table \ref{table5}. It is worth mentioning that the K-LED model~\eqref{pX} with $\beta = \left(1 + \frac{1}{K-1}\right) \cap \R_e$ (i.e., $K=2,4,6,8,10,...$). In fact, we have 
\begin{equation}\label{psiK}
\Psi(\theta) = \frac{K-1}{K}\left(\frac{\theta}{K-1} \right)^K
\end{equation}
If $K>2$ then $\nabla\Psi(0) = \nabla^2\Psi(0) = 0$ and thus \eqref{psiK} becomes the cumulant function of degenerate K-LED models, which are in the class of the parameterized log-concave density function. Note that the $2$-LED model with \eqref{psiK} (i.e., $\beta=2$) is the only classic probability distribution (normal distribution) having finite cumulants in Tweedie distribution. Additionally, as noticed in \cite{woo19b}, the regular Legendre transformation of Bregman-Tweedie divergence with \eqref{psiK} becomes an extended logistic loss function based on $\Psi$~\cite{woo19b}:
\begin{equation}\label{unh}
{\cal L}_{\Psi}(c,x) = \argmax_{z \in dom\Psi}\; \inprod{c}{z} - D_{\Psi}(z|x) = \nabla\Phi(c + \nabla\Psi(x))
\end{equation}
When $K=2$, \eqref{unh} is the unhinge loss function~\cite{rooyen15}. See also \cite{woo19} for the Perceptron-augmented extended logistic loss function. The following Theorem summarizes the mean parameter estimation of the K-LED model via Bregman divergence associated with $\Psi$ (or $\Psi^* = \Phi$).
\begin{theorem}\label{th:co}
Let $\{b_i\}_{i=1}^M$ be the observations. Here $b_i \in {\cal B} \subseteq \dom\Phi$, $b_{avg} = \frac{1}{M}\sum_{i=1}^M b_i \in \dom\Phi$, and $K \ge 1$. Consider the following maximization problem of the K-LED model with \eqref{powerlaw} in~\eqref{pX}. 
\begin{equation}\label{kltedmax}
\hat{\theta}_{\beta} = \argmax_{\theta \in dom\Psi} \sum_{i=1}^M \log p_{\Psi}(b_i; \theta,\sigma^2)
\end{equation}
Then, from Theorem \ref{canTH}, we have a unique solution
$\hat{\theta}_{\beta} = \ln_{2-\beta}\left(b_{avg} \right) \in \R_{\pm\infty}$
(an extended-valued real number system).
\end{theorem}

\subsection{The K-LED model with the convex extended logistic loss function\label{berpoi}}
This Section introduces a special K-LED model, of which the cumulant function is the convex extended logistic loss function~\cite{woo19}.

\begin{figure*}[t]
\centering
\includegraphics[width=6.5in]{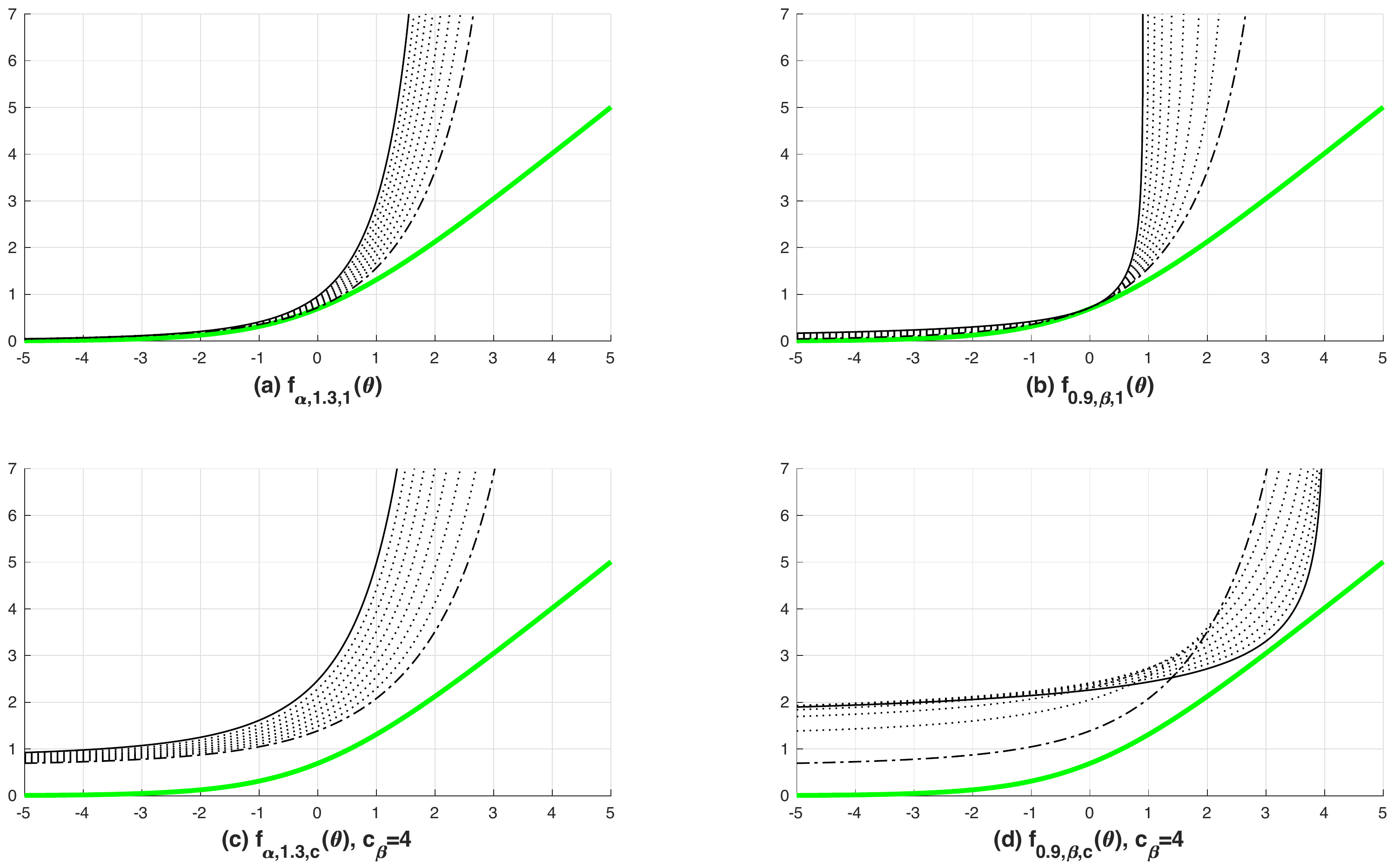} 
\caption{The graphs of the convex extended logistic loss function $f_{\alpha,\beta,c}(\theta) = \ln_{2-\alpha,c}(c + \exp_{2-\beta,c}(\theta))$. The green curve is the logistric loss: $f_{1,1,c}(\theta)$. (a) $f_{\alpha,0.7,1}(\theta)$ with $\alpha=1.9(\hbox{black}),1.8,...,1.2,1.1(\hbox{dash-dot})$. (b) $f_{1.1,\beta,1}(\theta)$ with $\beta=0.7(\hbox{dash-dot}),0.6,...,-0.1(\hbox{black})$. (c) $f_{\alpha,0.7,c}(\theta)$ with $\alpha=1.9(\hbox{black}),1.8,...,1.1(\hbox{dash-dot})$ and $c_{\beta}=4$. (d) $f_{1.1,\beta,c}(\theta)$ with $\beta=0.7(\hbox{dash-dot}),0.6,...,-0.1(\hbox{black})$ and $c_{\beta}=4$.
}
\label{fig:img2}
\end{figure*}

Let us consider the convex extended logistic loss function~\cite{woo19}:
\begin{equation}\label{xlog}
f_{\alpha,\beta,c}(\theta) = \ln_{2-\alpha,c}(c + \exp_{2-\beta,c}(\theta))
\end{equation}
where $c > 0$, $\beta \le \alpha \le 2$, and $\dom f_{\alpha,\beta,c} = \dom(\exp_{2-\beta,c})$. If $f_{\alpha,\beta,c} \in \L_1 \cap C^K(int(\dom f_{\alpha,\beta,c}))$ then we have a log-concave density function:
\begin{equation}\label{pfx}
p_{f_{\alpha,\beta,c}}(b;\theta,\sigma^2) = p_0(b;\sigma^2,\alpha,\beta,c)\exp(-d_{f^*_{\alpha,\beta,c}}(b;\theta)/\sigma^2) = p_1(b;\sigma^2,\alpha,\beta,c)\exp\left(\frac{\inprod{b}{\theta} - f_{\alpha,\beta,c}(\theta)}{\sigma^2} \right)
\end{equation}
where $p_1(b;\sigma^2,\alpha,\beta,c) = p_0(b;\sigma^2,\alpha,\beta,c)\exp(-f_{\alpha,\beta,c}^*(b)/\sigma^2).$  
\begin{theorem}\label{LegendreLogistic}
Consider the convex extended logistic loss function $f_{\alpha,\beta,c}(\theta)$~\eqref{xlog} with $\dom f_{\alpha,\beta,c} = \dom(\exp_{2-\beta,c})$. Then, under the following condition, $f_{\alpha,\beta,c} \in \L_1$ and is analytic on $\dom f_{\alpha,\beta,c}$.
\begin{itemize}
\item $\alpha = \beta  = 1$ and $\dom f_{\alpha,\beta,c} = \R$
\item $\alpha = 2$ and $\beta=1$ and $\dom f_{\alpha,\beta,c} = \R$
\item $\beta < \alpha \le 2$ and $\beta < 1$  and $\dom f_{\alpha,\beta,c} = \R_{<c_{2-\beta}}$
\end{itemize}
where $c_{2-\beta} = \frac{c^{\beta-1}}{1-\beta}$. If $(\alpha,\beta)=(1,1)$ then $p_{f_{\alpha,\beta,c}}(b;\theta,\sigma^2)$~\eqref{pfx} is  Bernoulli distribution.  If $(\alpha,\beta)=(2,1)$ and $c=1$, then it is Poisson distribution.
\end{theorem} 
\begin{proof}
\begin{itemize}
\item Let $\alpha=\beta=1$. Then $f_{1,1,c}(\theta) = \ln(1 + \exp(\theta)) \in \L_1 \cap C^{\infty}(\R)$ and thus we get Bernoulli distribution:
$
p_{f_{1,1,c}}(b;\theta,1) = p_1(b;1,1,1,c)\exp\left(\inprod{b}{\theta} - \ln(1 + \exp(\theta)) \right).
$
\item Let $\alpha=2$, $\beta=1$, and $c=1$. Then $f_{2,1,1}(\theta) = \exp(\theta) \in \L_1 \cap C^{\infty}(\R)$ and thus we get Poisson distribution: 
$
p_{f_{2,1,1}}(b;\theta,1) = p_1(b;1,2,1,1)\exp\left(\inprod{b}{\theta} - \exp(\theta) \right).
$
\item Let $\beta<\alpha \le 2$ and $\beta < 1$. Consider $f_{\alpha,\beta,c}(\theta) = \ln_{2-\alpha,c}(c + \exp_{2-\beta,c}(\theta))$
\begin{equation}\label{varx}
\nabla^2 f_{\alpha,\beta,c}(\theta) = \left[(\alpha-\beta) + (2-\beta)\frac{c}{\exp_{2-\beta,c}(\theta)}\right]\frac{\left(\exp_{2-\beta,c}(\theta)\right)^{4-2\beta}}{(c + \exp_{2-\beta,c}(\theta))^{3-\alpha}} > 0
\end{equation}
where $\exp_{2-\beta,c}(\theta) > 0$. Therefore, $f_{\alpha,\beta,c}$ is strictly convex on $\dom(\exp_{2-\beta,c}) = \R_{< c_{2-\beta}}$ which is open. From $\exp_{2-\beta,c}(\theta) = c(1-\theta/c_{2-\beta})^{\frac{1}{\beta-1}}$ with $\theta < c_{2-\beta}$ and $\beta<1$, we have $\exp_{2-\beta,c}(\theta) \in C^{\infty}(\R_{< c_{2-\beta}})$ and thus
$f_{\alpha,\beta,c} \in C^{\infty}(\dom f_{\alpha,\beta,c})$. Hence, we have a parameterized distribution between Bernoulli and Poisson:
$$
p_{f_{\alpha,\beta,c}}(b;\theta,\sigma^2) = p_1(b;\sigma^2,\alpha,\beta,c)\exp\left( \frac{\inprod{b}{\theta} - f_{\alpha,\beta,c}(\theta)}{\sigma^2} \right)
$$
where $p_1(b;\sigma^2,\alpha,\beta,c) = p_0(b;\sigma^2,\alpha,\beta,c)\exp(-f_{\alpha,\beta,c}^*(b)/\sigma^2)$ is an appropriate base measure satisfying $\int_{\cal B} p_{f_{\alpha,\beta,c}}(b;\theta,\sigma^2)\nu(db)=1$.
\end{itemize}
\end{proof}
From $\nabla f_{\alpha,\beta,c}(\theta) = \mu$, we get the mean of the probability distribution $p_{f_{\alpha,\beta,c}}$~\eqref{pfx}. Actually, when $\alpha=\beta=1$, $\mu = sigm(\theta)$ where $sigm(\theta)$ is a sigmoid function, and when $\alpha=2$ and $\beta=1$, we get $\mu = c\exp(\theta)$. However, when $\beta < 1$ and $\beta < \alpha \le 2$, we get a rather complicated mean
$
\mu = \frac{(\exp_{2-\beta,c}(\theta))^{2-\beta}}{(c + \exp_{2-\beta,c}(\theta))^{2-\alpha}}.
$
In terms of variance, by simple calculation, we have the variance 
$
var(b) = \sigma^2 \nabla^2 f_{\alpha,\beta,c}(\theta) > 0 
$ where $\nabla^2 f_{\alpha,\beta,c}(\theta)$ in \eqref{varx}. We do not a closed form of the unit  variance function $V(\mu)$ in terms of $\mu$, when $\beta<\alpha\le 2$ and $\beta<1$.

\section{Conclusion\label{sec6}}
This work introduces the K-LED model (Legendre exponential dispersion model). The cumulant function of the presented model is a convex function of Legendre type which has the continuous partial derivative of $K$-th order on its interior of the convex domain. The main advantage of the K-LED model is that it is generated by Bregman-divergence-guided log-concave density function having coercivity shape constraints, and thus the mean parameter space (or the first cumulant) is found by simple computation. Additionally, based on a subclass of the mean-variance relation of quasi-likelihood function, we build up the $2$-LED model. Under the mild regularity condition, we show the equivalence between quasi-likelihood function and the regular $2$-LED model. A typical example of the $2$-LED model is the extended Tweedie distribution with power variance function. This model is developed through  Bregman-Tweedie divergence ($\beta<0$) or Bregman-beta divergence ($\beta\ge0$). The base functions of these Bregman-divergences are induced from the extended exponential function or the extended logarithmic function. Notice that, by using the convex extended logistic loss function, we also show that a new parameterized K-LED model, which includes Bernoulli distribution and Poisson distribution, can be designed for the demonstration of the merit of the K-LED model. Last but not least, $3$-LED and $4$-LED could be further developed based on skewness and kurtosis~\cite{hy01}. 

\section*{Acknowledgments}
This paper is supported by the Basic Science Program through the NRF of Korea funded by the Ministry of Education (NRF-2015R101A1A01061261).

\section*{Appendix}
This section summarizes extended exponential and logarithmic functions and the corresponding convex functions of Legendre type which are introduced in \cite{woo19b}. 

Let us start with the definition of the extended exponential function~\cite{woo19b}. \begin{equation}\label{expraw}
\exp_{2-\beta,c}(x) = (c^{\beta-1} + (\beta-1)x)^{\frac{1}{\beta-1}}
\end{equation}
where $\exp_{1,c}(x) = c\exp(x)$ and $\dom(\exp_{2-\beta,c}) = \{ x \in \R \;|\; \exp_{2-\beta,c}(x) \in \R \}$. When $c=1$, we can recover the well-known generalized exponential function~\cite{amari00,tsallis09}. As observed in \cite{bar-lev86,woo17}, it is more convenient to use an equivalence class for the extended exponential function~\eqref{expraw}. 
\begin{definition}\label{def:exp}
Let $\beta \in \R$ and $x \in \dom(\exp_{2-\beta})$. Then the extended exponential function~\eqref{expraw} is simplified as
\begin{equation}\label{exexp}
\exp_{2-\beta}(x) \define
\left\{\begin{array}{l} 
\exp([x]), \qquad\qquad\quad\;\; \hbox{ if } \beta=1\\
((\beta-1)[x])^{1/(\beta-1)} , \quad \hbox{ otherwise }
\end{array}\right.
\end{equation}
where  $\dom(\exp_{2-\beta}) = \{ x \in \R \;|\; \exp_{2-\beta}(x) \in \R \}$ is in Table \ref{table3} and $c_{2-\beta} = \frac{c^{\beta-1}}{1-\beta} \in \R$. If $\beta \not= 1$, $[x] = x - c_{2-\beta}$ and if $\beta = 1$, $[x] = x + \ln(c)$. Note that, when $\beta<1$, $sign(c) = sign(\exp_{2-\beta}(x))$. 
\end{definition}
For simplicity, we use $x$, instead of the equivalence class $[x]$. Though we use an equivalence class~\eqref{exexp} for the extended exponential function~\eqref{expraw}, the role of $c$ (i.e., $c_{2-\beta}$) is important in machine learning. For instance, $c_{2-\beta}$ determines the margin of Hinge-Logitron, the loss function of which is the Perceptron-augmented extended logistic loss function. For more details, see \cite{woo19}. 

Now, consider the extended logarithmic function~\cite{woo17}:
\begin{equation}\label{lograw}
\ln_{2-\beta,c}(x) = c_{2-\beta} - x_{2-\beta}
\end{equation}
where $\ln_{1,c}(x) = \ln(x) - \ln(c)$. Note that, if $c=1$ then we recover the well-known generalized logarithmic function~\cite{amari00,tsallis09}. The proposed extended logarithmic function is also formulated with an equivalence class. That is, $\ln_{2-\beta}(x) =  [\ln_{2-\beta,c}(x)] = \ln_{2-\beta,c}(x) - c_{2-\beta}$. 
\begin{definition}\label{def:log}
Let $\beta \in \R$ and $x \in \dom(\ln_{2-\beta})$ then
\begin{equation}\label{exlog}
\ln_{2-\beta}(x) \define
\left\{\begin{array}{l} 
\ln(x), \qquad\quad\;\; \hbox{ if } \beta=1\\
\frac{1}{\beta-1}x^{\beta-1} , \qquad \hbox{ otherwise }
\end{array}\right.
\end{equation}
where $\dom(\ln_{2-\beta})$ is in Table \ref{table3}.
\end{definition}

\begin{table*}[h]
\centerline{\scriptsize
\begin{tabular}{c|c|ccc|ccc}
\hline\hline
   & $\beta = 1$ &  & $\beta > 1$ & & & $\beta < 1$ & \\ \cline{3-8} 
                 &   &  &  $\beta \in \R_e$ &  $\beta \in \R \setminus \R_e$   & &  $\beta \in \R_e$ &  $\beta \in \R \setminus \R_e$ \\ \hline 
 $\dom (\exp_{2-\beta})$   & $\R$  & & $\R$ & $\R_+$ &  & $\R_{--}$ / $\R_{++}$ & $\R_{--} $ \\ \hline
 $\dom (\ln_{2-\beta})$ & $\R_{++}$ &  & $\R$ & $\R_+$ & & $\R_{++}$ /  $\R_{--}$ & $\R_{++}$ \\ \hline\hline 
\end{tabular}}
\caption{The reduced domains of the extended exponential function $\exp_{2-\beta}$~\eqref{exexp} and the extended logarithmic function $\ln_{2-\beta}$~\eqref{exlog} for the bijection  $\exp_{2-\beta}= \ln_{2-\beta}^{-1} : \dom (\exp_{2-\beta}) \rightarrow \dom (\ln_{2-\beta})$. See \cite{woo19b} for more details. 
}\label{table3}
\end{table*}

The following theorem presents $\Psi$, the indefinite integral of the extended exponential function, satisfying the conditions of the convex function of Legendre type~\cite{roc70}. For more details, see \cite{woo19b}. Note that $\R_e = \{ 2k/(2l+1) \;|\; k,l \in \mathbb{Z} \}$.
\begin{theorem}\label{corPsi}
Let  $x \in \dom\Psi$ and $\exp_{2-\beta}$ be an extended exponential function in \eqref{exexp}. Then,
\begin{equation}\label{basefnD}
\Psi(x) = \int_d^{x} \exp_{2-\beta}(\xi)d\xi = 
\left\{\begin{array}{l} 
-\ln(-x) \qquad\quad\; \hbox{ if } \beta = 0 \\
\exp(x) \qquad\qquad\; \hbox{ if } \beta = 1 \\
\frac{1}{\beta}[(\beta-1)x]^{\frac{\beta}{\beta-1}} \;\; \hbox{otherwise }
\end{array}\right.
\end{equation}
is a convex function of Legendre type on $\dom  \Psi$:
\begin{equation}\label{conditionLegendreD}
\left\{
\begin{array}{l}
\hbox{I. entire region:}\\
\qquad 1<\beta,\; \beta \in \R_e  \hskip 0.75cm\hbox{ and } \hskip 0.2cm \dom\Psi = \R,\\
\qquad \beta=1,\; \hskip 1.8cm\hbox{ and } \hskip 0.1cm \dom\Psi = \R,\\
\hbox{II. positive region:}\\
\qquad 0 < \beta < 1, \beta \in \R_e \hskip 0.25cm \hbox{ and }  \hskip 0.1cm  \dom\Psi = \R_{++},\\
\qquad \beta < 0, \qquad \beta \in \R_e \hskip 0.2cm \hbox{ and }  \hskip 0.1cm  \dom\Psi = \R_{+},\\
\hbox{III. negative region:}\\
\qquad 0 \le \beta < 1,  \hskip 1.25cm\hbox{ and }  \hskip 0.1cm \dom\Psi = \R_{--},\\
\qquad \beta<0,  \hskip 1.88cm \hbox{ and }  \hskip 0.1cm \dom\Psi = \R_{-}. \\
\end{array}
\right.
\end{equation}
Here, all constant terms are dropped.
\end{theorem}
In consequence, we can define {\it Bregman-Tweedie divergence} (i.e., Bregman divergence associated with $\Psi$) as  
\begin{equation}\label{tweedieBregman}
D_{\Psi}(x|y) = \Psi(x) - \Psi(y) - \inprod{\nabla\Psi(y)}{x-y}
\end{equation}
where $(x,y) \in \dom\Psi \times int(\dom\Psi)$. Here, $\dom\Psi$ is in \eqref{conditionLegendreD}. Note that $\Phi(x)= \int_d^x \ln_{\alpha}(t)dt$ was studied in \cite{woo17} for the characterization of $\beta$-divergence based on Bregman divergence framework. {\it Bregman-beta divergence} (i.e., Bregman divergence associated with $\Phi$) is defined by
\begin{equation}\label{betaBregman}
D_{\Phi}(x|y) = \Phi(x) - \Phi(y) - \inprod{\nabla\Phi(y)}{x-y}
\end{equation}
where $\Phi$ is introduced in the following theorem.
\begin{theorem}\label{legendre_beta}
Let $x \in dom\Phi$ and 
\begin{eqnarray}\label{basefn}
\Phi(x) = \int_d^{x} \ln_{2-\beta}(t) dt
= \left\{\begin{array}{l} 
 -\log x, \quad\;\;\;\;\; \hbox{ if } \beta=0,\\ 
x\log x - x, \;\;\; \hbox{ if } \beta =1,\\
\frac{1}{\beta(\beta-1)}x^{\beta}, \;\;\;\;\;\; \hbox{otherwise }
\end{array}\right.
\end{eqnarray}
where  $\ln_{2-\beta}(t)$ is the extended logarithmic function in Definition \ref{def:log}. Then, $\Phi$ in \eqref{basefn} is a convex function of Legendre type on $\dom\Phi$ below: 
\begin{equation}\label{conditionLegendre}
\left\{
\begin{array}{l}
\hbox{I. entire region:}\\
\qquad 1<\beta,\; \beta \in \R_e  \hskip 0.85cm\hbox{ and } \hskip 0.1cm \dom\Phi = \R,\\
\hbox{II. positive region:}\\
\qquad  0 <  \beta \le 1 \hskip 1.55cm \hbox{ and }  \hskip 0.1cm \dom\Phi  = \R_{+},\\
\qquad \beta \le 0  \hskip 2.2cm \hbox{ and }  \hskip 0.1cm  \dom\Phi = \R_{++},\\
\hbox{III. negative region:}\\
\qquad 0< \beta < 1, \; \beta \in \R_e  \hskip 0.3cm\hbox{ and }  \hskip 0.1cm \dom\Phi = \R_{-},\\
\qquad \beta < 0 , \;\qquad \beta \in \R_e \hskip 0.3cm \hbox{ and }  \hskip 0.1cm \dom\Phi = \R_{--}. \\
\end{array}
\right.
\end{equation}
For simplicity, all constants in $\Phi(x)$ are dropped. 
\end{theorem}
As observed in \cite{woo17}, due to the invariance properties with respect to the affine function of the base function of Bregman divergence, the structure of Bregman-beta divergence $D_{\Phi}$ does not change, irrespective of the choice of the affine function in the base function. Hence, for simplicity, we add $x$ to $\Phi$~\eqref{basefn}, when $\beta=1$.


\begin{thebibliography}{30}
\bibitem{amari00}{\sc S. Amari and H. Nagaoka,} {\em Methods of Information Geometry}, AMS, 2000.
\bibitem{amari16}{\sc S. Amari,} {\em Information geometry and its applications}, Springer, 2016.
\bibitem{banerjee05}{\sc A. Banerjee, S. Merugu, I. Dhillon, and J. Ghosh,} ``Clustering with Bregman Divergences'', {\em J. of Mach. Learn. Res.}, 6 (2005), pp. 1705-1749.
\bibitem{barndorff14}{\sc O. Barndorff-Nielsen,} {\em Information and Exponential Families in Statistical Theory}, Wiley, 2014.
\bibitem{bar-lev86}{\sc S. Bar-Lev and P. Enis,} ``Reproducibility and natural exponential families with power variance functions'', {\em The Annals of Statistics}, 14 (1986), pp. 1507-1522.
\bibitem{basbug17}{\sc M. Basbug and B. Engelhardt,} ``AdaCluster: adaptive clustering for heterogeneous data", {\em arXiv:1510.05491v2} (2017), pp. 1-34.
\bibitem{basu98}{\sc A. Basu, I. Harris, N. Hjort, and M. Jones,} ``Robust and efficient estimation by minimizing a density power divergence'', {\em Biometrika}, 85 (1998), pp. 549-559.
\bibitem{bauschke97}{\sc H. Bauschke and J. Borwein,} ``Legendre functions and the method of random Bregman projections", 
{\em Journal of Convex Analysis}, 4 (1997), pp. 27-67.
\bibitem{bauschke11}{\sc H. Bauschke and P. Combettes,} {\em Convex Analysis and Monotone Operator Theory in Hilbert Spaces}, (2011), Springer.
\bibitem{bregman67}{\sc L. Bregman,} ``The relaxation method of finding the common points of convex sets and its application to the solution of problems in convex programming", {\em USSR Computational Mathematics and Mathematical Physics}, 7 (1967), pp. 200-217.
\bibitem{brown86}{\sc L. Brown,} {\em Fundamentals of statistical exponential families with applications in statistical decision theory}, Institute of Mathematical Statistics, Hayworth, CA, USA, (1986). 
\bibitem{cichocki10}{\sc A. Cichocki and S. Amari,} ``Families of Alpha- Beta- and Gamma- divergences: Flexible and robust measures of similarities'', {\em Entropy}, 12 (2010), pp. 1532-1568.
\bibitem{cule10}{\sc M. Cule and R. Samworth,} ``Theoretical properties of the log-concave maximum likelihood estimator of a multidimensional density'', {\em Elec. Journal of Stat.}, 4 (2010), pp. 254-270.
\bibitem{samworth18}{\sc R. Samworth,}
``Recent progress in log-concave density estimation'', {\em Statist. Sci.}, 33 (2018), pp. 493-509.
\bibitem{depierro86}{\sc A. de Pierro and A. Iusem,}  ``A relaxed version of Bregman's method for convex programming'', {\em J. of Optimization Theory and Applications}, 51 (1986), pp. 421-440.

\bibitem{dikmen15}{\sc O. Dikmen, Z. Yang, and E. Oja,} ``Learning the information divergence", {\em IEEE Trans. Pattern Recognition and Machine Learning}, 37 (2015), pp. 1442-1454.

\bibitem{eguchi01}{\sc S. Eguchi and Y. Kano,} ``Robustifying maximum likelihood estimation'', {\em Technical Report}, Institute of Statistical Mathematics, June 2001.

\bibitem{dias10b}{\sc M. Figueiredo and J. Bioucas-Dias,} ``Restoration of Poissonian images using alternating direction optimization", {\em IEEE Trans. Image Processing}, 19 (2010), pp. 3133-3145.

\bibitem{fevotte09}{\sc C. Fevotte, N. Bertin, and J.-L. Durrieu,} ``Nonnegative Matrix Factorization with the Itakura-Saito Divergence: with application to music analysis",
{\em Neural Computation}, 21 (2009), pp. 793-830.

\bibitem{jiang12}{\sc K. Jiang, B. Kulis, and M. Jordan,} ``Small-variance asymptotics for exponential family Dirichlet process mixture models", {\em Proc. on Neural Information Processing Systems}, 25 (2012).
 
\bibitem{jorgensen97}{\sc B. Jorgensen,} {\em The Theory of Dispersion Models}, Chapman \& Hall, 1997.

\bibitem{halmos49}{\sc P. Halmos and L. Savage,} ``Application of the Radon-Nikodym theorem to the theory of sufficient statistics", {\em Ann. Math. Statist.} 20 (1949), pp. 225-241.

\bibitem{hir96}{\sc J.-B. Hiriart-Urruty and C. Lemarechal,} {\em Convex Analysis and Minimization Algorithms I, II}, Springer-Verlag, 1996.

\bibitem{hy01}{\sc A. Hyv\"arinen, J. Karhunen, and E. Oja,} {\em Independent component analysis}. Wiley Interscience, 2001.

\bibitem{kass97}{\sc R. Kass and P. Vos,} {\em Geometrical foundations of asymptotic inference}. Wiley Interscience, 1997.

\bibitem{kulis12}{\sc B. Kulis and M. Jordan,}  ``Revisiting k-means: New algorithms via bayesian nonparametrics", {\em ICML}, 2012.

\bibitem{lecellier10}{\sc F. Lecellier, J. Fadili, S. Jehan-Besson, G. Aubert, M. Revenu, and  E. Saloux,} ``Region-based active contours with exponential family observations", {\em J. Math. Imaging Vis.}, 36 (2010), pp. 28-45.

\bibitem{lukacs58}{\sc E. Lukacs,} ``Some extensions of a theorem of Marcinkiewicz", {\em Pacific J. Math.}, 8 (1958), pp. 487-501. 

\bibitem{mccullagh89}{\sc P. McCullagh and J. Nelder,} {\em Generalized Linear Models}, Chapman\& Hall/CRC, 1989.

\bibitem{morris82}{\sc C. Morris,} ``Natural exponential families with quadratic variance functions'', {\em The Annals of Statistics}, 10 (1982), pp. 65-80.

\bibitem{murphy12}{\sc K. Murphy,} {\em Machine Learning}, MIT Press, 2012.

\bibitem{muller02}{\sc H.-J. M\"uller, R. Horn, and A. Moreira,} ``From Gaussian to inverse Gaussian statistics in SAR imagery,'' {\em Proc. IGARSS}, (2002), pp. 2486-2488. 

\bibitem{nielsen09}{\sc F. Nielsen and R. Nock,} ``Sided and symmetrized Bregman centroids", 
{\em IEEE Trans. Information Theory} 55 (2009), pp.2882-2904.

\bibitem{paul13}{\sc G. Paul, J. Cardinale, and I. Sbalzarini,} ``Coupling image restoration and segmentation: a generalized linear model/Bregman perspective", {\em Int. J. Comput. Vis.}, 104 (2013), pp. 69-93.

\bibitem{pistone99}{\sc G. Pistone and H. Wynn,} ``Finitely generated cumulants", {\em Statistica Sinica}, 9 (1999), pp. 1029-1052.

\bibitem{reid10}{\sc M. Reid and R. Williamson,} "Composite Binary Losses", 
{\em J. of Mach. Learn. Res.} 11 (2010), pp.2387-2422.

\bibitem{reid11}{\sc M. Reid and R. Williamson,} ``Information, divergence and risk for binary experiments", 
{\em J. of Mach. Learn. Res.} 12 (2011), pp.731-817.

\bibitem{roc70}{\sc R. Rockafellar,} {\em Convex Analysis},
Princeton University Press, Princeton, 1970.

\bibitem{rudin92}{\sc L. Rudin, S. Osher, and E. Fatemi,} ``Nonlinear total variation based noise removal algorithms", {\em Physica D}, 60 (1992), pp. 259-268.

\bibitem{teboulle07}{\sc M. Teboulle,} ``A unified continuous optimization framework for center-based clustering methods'', {\em J. of Mach. Learn. Res.}, 8 (2007), pp. 65-102.

\bibitem{tsallis09}{\sc C. Tsallis,} {\em Introduction to nonextensive statistical mechanics: approaching a complex world}, Springer, 2009.

\bibitem{samek13}{\sc W. Samek, D. Blythe, K.-R. Muller, and M. Kawanabe,} ``Robust spatial filtering with beta divergence'', {\em Proc. on Neural Information Processing Systems}, (2013), pp. 1007-1015.

\bibitem{saumard14}{\sc A. Saumard and J. Wellner,} ``Log-concavity and strong log-concavity: A review", {\em Statistics Surveys}, 8 (2014), pp.45-114.

\bibitem{tweedie84}{\sc M. Tweedie,}  ``An index which distinguishes between some important exponential families", {\em Proc. Indian Stat. Inst. Golden Jubilee International Conference}, (1984), pp. 579–604.

\bibitem{rooyen15}{\sc B. van Rooyen, A. Menon, and R. Williamson,}  ``Learning with symmetric label noise: The importance of being unhinged", {\em Proc. on Neural Information Processing Systems}, (2015), pp. 10-18.

\bibitem{wainwright08}{\sc M. Wainwright and M. Jordan,} ``Graphical models, Exponential families, and variational inference'', {\em Foundations and Trends in Machine Learning}, 1 (2008), pp. 1-305.

\bibitem{wedderburn74}{\sc R. Wedderburn,} ``Quasi-likelihood functions, generalized linear models, and the Gauss-Newton method'', {\em Biometrika}, 61 (1974), pp. 439-447.

\bibitem{woob16}{\sc H. Woo,} ``Beta-divergence based two-phase segmentation model for synthetic aperture radar images'', {\em Electronics Letters}, 52 (2016), pp. 1721-1723.

\bibitem{woo17}{\sc H. Woo,} ``A characterization of the domain of Beta-divergence and its connection to Bregman variational model", {\em Entropy}, 19 (2017), 482.

\bibitem{woo19}{\sc H. Woo,} ``Logitron: Perceptron-augmented classification framework based on extended logistic loss function", {\em  arXiv: 1904.02958}, (2019), pp. 1-16.

\bibitem{woo19b}{\sc H. Woo,} ``The Bregman-Tweedie classification model", {\em arXiv: 1907.06923v1}, (2019), pp. 1-21.

\end{thebibliography}
\end{document}